\documentclass[english,10pt]{article}
\usepackage[T1]{fontenc}
\usepackage[latin9]{inputenc}
\usepackage{amsmath}
\usepackage{amssymb}
\usepackage{amsfonts}
\usepackage{babel}
\usepackage{fancyhdr}
\usepackage{color}
\usepackage[toc]{appendix}
\usepackage{graphicx}

\makeatletter
\newtheorem{theorem}{Theorem}

\newtheorem{example}{Example}

\newtheorem{proposition}[theorem]{Proposition}
\newtheorem{remark}[theorem]{Remark}

\newenvironment{proof}[1][Proof]{\noindent\textbf{#1.} }{\ \rule{0.5em}{0.5em}}
\makeatother
\begin{document}

\title{Distinguishing and integrating aleatoric and epistemic variation in
uncertainty quantification}
\author{Kamaljit Chowdhary\thanks{%
Division of Applied Mathematics, Brown University, Providence, RI, 02912,
USA. Research supported in part by the Air Force Office of Scientific
Research (FA9550-07-1-0425).} \thinspace and Paul Dupuis\thanks{%
Lefschetz Center for Dynamical Systems, Division of Applied Mathematics,
Brown University, Providence, RI, 02912, USA. Research supported in part by
the National Science Foundation (DMS-1008331), the Department of Energy
(DE-SCOO02413), and the Air Force Office of Scientific Research
(FA9550-07-1-0425 and FA9550-09-1-0378).}}
\maketitle

\begin{abstract}
Much of uncertainty quantification to date has focused on determining the
effect of variables modeled probabilistically, and with a known
distribution, on some physical or engineering system. We develop methods to
obtain information on the system when the distributions of some variables
are known exactly, others are known only approximately, and perhaps others
are not modeled as random variables at all. \ The main tool used is the
duality between risk-sensitive integrals and relative entropy, and we obtain
explicit bounds on standard performance measures (variances, exceedance
probabilities) over families of distributions whose distance from a nominal
distribution is measured by relative entropy. The evaluation of the
risk-sensitive expectations is based on polynomial chaos expansions, which
help keep the computational aspects tractable.
\end{abstract}

\section{Introduction}

\textit{Uncertainty quantification} refers to a broad set of techniques for
understanding the impact of uncertainties in complicated mechanical and
physical systems. In this context {}\textquotedblleft
uncertainty\textquotedblright {}\ can take on many meanings, but we follow
the convention of dividing uncertainty into aleatoric and epistemic
categories. In short, aleatoric uncertainty refers to inherent uncertainty
due to stochastic or probabilistic variability. This type of uncertainty is 
\emph{irreducible} in that there will always be positive variance since the
underlying variables are truly random. Epistemic uncertainty refers to
limited knowledge we may have about the model or system. This type of
uncertainty is \emph{reducible} in that if we have more information, e.g.,
take more measurements, then this type of uncertainty can be reduced.
However, for many problems where uncertainty quantification is important,
the acquisition of data is difficult or expensive. The epistemic uncertainty
cannot be removed entirely, and so one needs modeling and computational
techniques which can also accommodate this form of uncertainty.

Much of the work to date has focused on aleatoric uncertainty and what one
might call the propagation of uncertainty. Here, one characterizes
uncertainty concerning one or more elements of a physical system via some
probability distribution, and attempts to quantify how that uncertainty will
be propagated throughout the system under its constitutive equations and
laws. The simplest and most straightforward method to characterize this
output distribution would be to sample from the input probability
distribution, solve the system equations, and thereby produce output samples
(i.e., standard Monte Carlo). Although this is very simple conceptually, it
can be far from practical, especially when it is computationally intensive
to construct the mapping that takes input variables to output.
Straightforward schemes such as Monte Carlo would require that the system be
solved for each sample of the random variable, a situation that is often not
practical. In recent years efficient computational methods have been
developed to calculate particular functionals of the induced distribution
that may be required for a particular application (e.g., covariances or
error probabilities). Most recently, polynomial chaos has become popular as
an efficient method for approximating the distribution of the output
variable.

We note that random variables with a well-defined and given distribution are
often used in this context even when there is no justification for their
use, as in the case of modeling error. In other words, it is (perhaps
implicitly) assumed that epistemic uncertainty can be modeled by aleatoric
uncertainty. One reason is that, at least as they have been developed to
date, most computational techniques (e.g., polynomial chaos and Monte Carlo)
are based on the assumption that the user can identify some
{}\textquotedblleft appropriate\textquotedblright {}\ distribution for each
uncertain aspect of the system, regardless of the type of uncertainty,
aleatoric or epistemic. If one is interested in just basic qualitative
properties of the system then this may not be a central issue, since
virtually any model of uncertainty will give information on the
sensitivities of the system. However, when the intended use of uncertainty
quantification is for regulatory assessment or some other application where
performance measures are sensitive to distributional assumptions, the issue
becomes much more important, and one should carefully distinguish how one
accounts for the two types of uncertainty.

The aim of the present paper is to describe an approach that (i) logically
distinguishes those aspects of uncertainty that are treated as stochastic
variability from other forms of uncertainty, (ii) in cases where a
stochastic model is theoretically valid but for which determination of the
distribution is not practical, gives bounds for performance measures that
are valid for explicitly identified families of distributions, and (iii) is
computationally feasible if ordinary uncertainty propagation is feasible.\
The intended audience is broad, including numerical analysts interested in
robust performance bounds as well as applied probabilists for whom certain
computational aspects of uncertainty quantification may be novel. Since a
typical reader might not be familiar with the terminology and methods from
both fields, we have included background material for both the probabilistic
and numerical analysis approaches used to make the paper broadly accessible.

The paper is organized as follows. In Section 2 we further discuss the
distinction between aleatoric and epistemic uncertainty, and explain some of
the limitations to achieving useful performance bounds in the presence of
epistemic uncertainty. Section 3 introduces risk-sensitive performance
measures and discusses their robust properties. Several hybrid forms are
introduced that are more useful than the simplest form when both aleatoric
and epistemic uncertainties are present. Various properties of the
risk-sensitive measures that are useful in applications (monotonicity,
optimization) are also discussed. Section 4 presents numerical examples,
reviews the computational methods used to evaluate the risk-sensitive
performance measures, and finishes with a subsection of discussion and
conclusions.

\section{Aleatoric and epistemic uncertainties}

As noted in the Introduction, \emph{uncertainty} can be divided into two
categories, aleatoric and epistemic. From the perspective of mathematical
formulation and modeling, aleatoric uncertainty is in some sense simpler. In
contrast, epistemic uncertainty can mean different things in different
contexts. \textit{Lack of knowledge} is an ambiguous term that encompasses
many different scenarios.

To illustrate, we consider an elementary example. Consider a system
described by some well-posed partial differential equation (PDE), but with
an uncertain boundary condition. The boundary condition is expressed in
terms of a random variable $X$ (e.g., $u_{x}(0,t)=X$) with a particular
fixed distribution, and the numerical analysis problem is to characterize
the induced distribution of the solution to the PDE at some time and
location (e.g., $u(x_{0},t_{0};X)$). Suppose we know that the boundary
condition is properly modeled as a random variable, but we do not know the
correct value of some parameter in that distribution. For example, based on
a central limit type argument, one might claim the distribution is known to
be Gaussian, but still unknown are the \textquotedblleft
true\textquotedblright\ mean and/or variance. In this case, the \emph{lack
of knowledge} is this missing information about the parameters of the
distribution. The lack of information regarding these parameters is a form
of epistemic uncertainty. Hence in this example aleatoric and epistemic
uncertainty are mingled. Assuming that samples are available, one could use
the empirical mean and variance from data to approximate the true mean and
variance, and thus reduce this uncertainty. However, in the common situation
where sampling is necessarily limited, some epistemic uncertainty is
inherent in the model due to practical limitations on data acquisition and
modeling.

Another type of epistemic uncertainty is the omission of important aspects
of the system model. This form is possibly the most difficult to quantify.
For example, there might be a hidden random variable in the model. In the
PDE example we have assumed the boundary condition is modeled via a random
variable, but there may be other coefficients in the model that are random
but treated as constants. It is also frequently true that the mathematical
model used for the system is only approximately true. For example, the model
might simplify the geometry of the true system, nonlinearities may be
approximated by linear relations, etc. In the context of the PDE example,
boundary conditions of Dirichlet form were chosen, when in fact a more
realistic model might use a mixed form (e.g., Robin boundary conditions).

For almost all forms of epistemic uncertainty there is little justification
for the use of random variables to model the uncertainty. Nonetheless, it is
common practice to do just that, and in practice randomness and random
variables are used in many situations to account for errors in the physical
model or modeling ignorance. The reasons were mentioned previously--that
basic distributional properties of the output (e.g., variance) will still
provide a reasonable sensitivity analysis for the problem, and existing
computational methods are largely based on aleatoric uncertainty. \ However,
with little justification for the use of any particular distribution (or
even the use of randomness at all), in more critical applications one may
insist on rigorous bounds on performance that are valid for a particular
family of distributions. In the extreme case where no probabilistic model is
considered acceptable, one may wish to only prescribe bounds on certain
parameters, and then obtain tight bounds on performance over all values of
the parameters that satisfy the bounds.

Thus there are many different types of uncertainty that one should account
for in an analysis of the effects of uncertainty. These include: (i)
aleatoric with known distribution; (ii) aleatoric with partly known
distribution (mingled aleatoric and epistemic); (iii) epistemic for which
one is willing to model by a family of aleatoric uncertainties, and (iv)
epistemic where one is only willing to place bounds on the uncertainties. As
remarked in the Introduction, this paper will introduce an approach that
allows these uncertainties to be handled within a single framework that can
exploit computational methods originally developed just for the treatment of
aleatoric uncertainties.

\section{Duality for exponential integrals}

Our development of performance measures that distinguish forms of
uncertainty, and which in particular allow for robustness with respect to
epistemic uncertainties, depends on a duality relation between exponential
integrals and relative entropy. Its first use along these lines but with
regard to estimation appears to be in \cite{boejampet}, and for optimization
in \cite{dupjampet}. We first state the basic duality result, and then
define two functionals which will allow aleatoric and epistemic
uncertainties to be analyzed simultaneously but at the same time
differentiated.

The general duality is stated for random variables that take values in a
Polish space (i.e., a complete, separable metric space) $\mathcal{X}$. The
associated $\sigma $-algebra is the Borel $\sigma $-algebra. A typical
example of $\mathcal{X}$ for our purposes is a closed subset of some
Euclidean space $\mathbb{R}^{d}$. Let $\mathcal{P}(\mathcal{X})$ denote the
collection of probability measures on $\mathcal{X}$, and let $\mu \in $ $%
\mathcal{P}(\mathcal{X})$. Given any $\nu \in $ $\mathcal{P}(\mathcal{X})$
that is absolutely continuous with respect to $\mu $, we define the relative
entropy of $\nu $ with respect to $\mu $ by%
\begin{equation*}
R\left( \nu \left\Vert \mu \right. \right) =\int_{\mathcal{X}}\log \left( 
\frac{d\nu }{d\mu }(x)\right) \nu (dx)
\end{equation*}%
whenever $\log (d\nu /d\mu (x))$ is integrable with respect to $\nu $. In
all other cases $R\left( \nu \left\Vert \mu \right. \right) $ is defined to
be $\infty $.

Relative entropy defines a mapping $(\nu ,\mu )\rightarrow R\left( \nu
\left\Vert \mu \right. \right) $ from $\mathcal{P}(\mathcal{X})^{2}$ to $%
\mathbb{R}\cup \infty $. This mapping has a number of very attractive
properties. For example, relative entropy is non-negative, with $R\left( \nu
\left\Vert \mu \right. \right) =0$ if and only if $\nu =\mu $. In addition,
the mapping is jointly convex and lower semicontinuous \cite[Lemma 1.4.3]%
{dupell4}. A property of particular interest for our purposes is the
following variational formula for exponential integrals, which can be
considered as an infinite dimensional version of the Legendre transform.
Given any bounded and continuous function $F:\mathcal{X}\rightarrow \mathbb{R%
}$ and any $c\in (0,\infty )$, 
\begin{equation}
\Lambda _{c}\doteq \frac{1}{c}\log \int_{\mathcal{X}}e^{cF(x)}\mu
(dx)=\sup_{\nu \in \mathcal{P}(\mathcal{X})}\left[ -\frac{1}{c}R\left( \nu
\left\Vert \mu \right. \right) +\int_{\mathcal{X}}F(x)\nu (dx)\right] .
\label{RE_duality}
\end{equation}%
For a proof see \cite[Proposition 1.4.2]{dupell4}. The duality continues to
hold as stated if $F$ is bounded from above, and also when bounded from
below if one restricts the supremum to $\nu \in \mathcal{P}(\mathcal{X})$
for which $R\left( \nu \left\Vert \mu \right. \right) <\infty $ \cite[%
Proposition 4.5.1]{dupell4}.

As an elementary example on how such a formula can be useful, suppose that $%
F $ is a performance measure (e.g., variance or an error probability), and
that $\mu $ is a model for some random phenomena. We consider $\mu $ to be
the \textit{nominal model}, e.g., our best guess. However, we are uncertain
if this is the correct model, and would like a measure of performance that
also holds for alternative models, but with a penalty for deviation from the
nominal model. $\Lambda _{c}$ is just such a performance measure, since by (%
\ref{RE_duality}), for any alternative model $\nu $ the bound 
\begin{equation*}
\int_{\mathcal{X}}F(x)\nu (dx)\leq \Lambda _{c}+\frac{1}{c}R\left( \nu
\left\Vert \mu \right. \right)
\end{equation*}%
applies. Hence we obtain bounds on performance over a family of alternative
models. The parameter $c$ allows one to balance robustness with respect to
possible model inaccuracies against tighter bounds. In particular, as $c$
tends to $0$, $\Lambda _{c}$ converges to $\int_{\mathcal{X}}F(x)\mu (dx)$,
which is the performance measure under the nominal model, but in the limit
the bound is meaningful only when $\nu =\mu $.

Suppose that a bound on performance over a specific family of distributions
is needed. Let $R^{\ast }$ denote the maximum of relative entropy with
respect to the nominal model over this family. Then the tightest possible
bound is obtained by minimizing 
\begin{equation*}
\Lambda _{c}+\frac{1}{c}R^{\ast }
\end{equation*}%
over $c>0$. We show in Proposition \ref{prop2} below that this function has
only one local minimum over $c\in (0,\infty ]$, and thus the global minimum
is easy to compute.

Functionals of the exponential form that appears in the definition of $%
\Lambda _{c}$ are sometimes called \textit{risk-sensitive} because the
exponential amplifies the effect of any large values of the performance
measure $F$. In the next two sections we consider two different hybrid
risk-sensitive functionals that will allow aleatoric and epistemic variables
to be combined and yet differentiated.

\subsection{Two hybrid forms}

\label{Section:RSforms}

In this section and the next random variables with a \emph{known}
distribution will take values in a Polish space $\mathcal{X}$. Variables
whose distribution is \emph{not known} or are otherwise of the epistemic
variety take values in the space $\mathcal{Y}$.

The performance measure of interest for some given problem is assumed to be
of the form%
\begin{equation}
\int_{\mathcal{X}}\int_{\mathcal{Y}}F(x,y)\gamma (dy)\mu (dx),
\label{std_cost}
\end{equation}%
where $\mu $ (resp., $\gamma $) is a probability measure on $\mathcal{X}$
(resp., $\mathcal{Y}$). If $X$ and $Y$ are independent random variables with
distributions $\mu $ and $\gamma $, then $F(X,Y)$ represents both the
performance measure (e.g., a second moment) as well as the underlying
physical or mechanical system that maps these aleatoric and epistemic inputs
into outputs. The analogous ordinary risk-sensitive performance measure is 
\begin{equation}
\Lambda _{c}=\frac{1}{c}\log \int_{\mathcal{X}}\int_{\mathcal{Y}%
}e^{cF(x,y)}\gamma (dy)\mu (dx).  \label{rs_cost}
\end{equation}%
Neither of the measures (\ref{std_cost}) or (\ref{rs_cost}) differentiate
the variables according to type (aleatoric or epistemic) and given that the
performance measure of interest is actually $F$, the use of a risk-sensitive
version of the cost is not well motivated for the aleatoric variables.
Indeed, use of this measure will give bounds that are robust with respect to
variations on a distribution that is known, and obviously such bounds will
not be as tight as possible.

The first form we consider for a hybrid measure is%
\begin{equation}
\Lambda _{c}^{1}=\frac{1}{c}\log \int_{\mathcal{Y}}e^{\int_{\mathcal{X}%
}cF(x,y)\mu (dx)}\gamma (dy).  \label{hybrid1}
\end{equation}%
Note that by Jensen's inequality (applied to the convex function $\exp $),
this is smaller than (and in general strictly smaller than) $\Lambda _{c}$.
Using the relative entropy representation for exponential integrals given in
(\ref{RE_duality}) [but with $\mathcal{Y}$ in place of $\mathcal{X}$ and
letting $\int_{\mathcal{X}}cF(x,y)\mu (dx)$ be the cost function], it
follows that for any distribution $\theta (dy)$ 
\begin{equation}
\int_{\mathcal{Y}}\int_{\mathcal{X}}F(x,y)\mu (dx)\theta (dy)\leq \frac{1}{c}%
R\left( \theta (dy)\left\Vert \gamma (dy)\right. \right) +\Lambda _{c}^{1}.
\label{first}
\end{equation}%
This gives a bound on the performance measure for an arbitrary distribution
on $Y$, but with the distribution on $X$ equal to the known true
distribution. The distributions thus play very different roles. In
particular, we think of $\gamma $ as a \emph{nominal }distribution of $Y$,
which should be distinguished from a possible \textit{true} distribution.
The risk sensitive functional $\Lambda _{c}^{1}$, whose numerical evaluation
can be carried out by a variety of methods (including polynomial chaos
expansion as discussed below), is based on the nominal distribution. Through
the relative entropy duality, it will yield various bounds (depending on $c$%
) on a families of distributions, with the relative entropy distance the key
metric. Ideally one would consider the smallest family that contains the
\textquotedblleft true\textquotedblright\ distribution (if it exists), and
get the tightest bound by optimizing over $c$. With the hybrid
risk-sensitive formulation, $\mu $ is allowed to be both the nominal
(computational) distribution and also the true distribution. \ 

The second form one could consider for a hybrid measure is 
\begin{equation}
\Lambda _{c}^{2}=\frac{1}{c}\int_{\mathcal{X}}\left[ \log \int_{\mathcal{Y}%
}e^{cF(x,y)}\gamma (dy)\right] \mu (dx).  \label{hybrid2}
\end{equation}%
Note that by Jensen's inequality (applied now to the concave function $\log $%
), this is again in general strictly smaller than $\Lambda _{c}$. The
relative entropy representation for exponential integrals gives 
\begin{equation*}
\int_{\mathcal{Y}}F(x,y)\theta (dy|x)\leq \frac{1}{c}R\left( \theta
(dy|x)\left\Vert \gamma (dy)\right. \right) +\frac{1}{c}\log \int_{\mathcal{Y%
}}e^{cF(x,y)}\gamma (dy).
\end{equation*}%
Here $\theta (dy|x)$ is any stochastic kernel on $\mathcal{Y}$ given $%
\mathcal{X}$, which is essentially a conditional distribution on $\mathcal{Y}
$ given $X=x$ (for the precise definition see \cite[page 35]{dupell4}).
Integrating over $\mathcal{X}$ gives 
\begin{equation*}
\int_{\mathcal{X}}\int_{\mathcal{Y}}F(x,y)\theta (dy|x)\mu (dx)\leq \frac{1}{%
c}\int_{\mathcal{X}}R\left( \theta (dy|x)\left\Vert \gamma (dy)\right.
\right) \mu (dx)+\Lambda _{c}^{2}.
\end{equation*}%
In the case where $\theta (dy|x)$ is independent of $x$, we obtain%
\begin{equation*}
\int_{\mathcal{X}}\int_{\mathcal{Y}}F(x,y)\theta (dy)\mu (dx)\leq \frac{1}{c}%
R\left( \theta (dy)\left\Vert \gamma (dy)\right. \right) +\Lambda _{c}^{2}.
\end{equation*}

Note that the second hybrid cost gives a more general bound, in that it
allows the alternative distribution on $Y$ (i.e., $\theta (dy|x)$) to depend
on the value taken by $X$. This additional flexibility comes at a cost, and
indeed we will show in the next subsection that the following inequalities
hold (typically in a strict fashion):

\begin{equation}
\int_{\mathcal{X}}\int_{\mathcal{Y}}F(x,y)\gamma (dy)\mu (dx)\leq \Lambda
_{c}^{1}\leq \Lambda _{c}^{2}\leq \Lambda _{c}.  \label{inequalities}
\end{equation}%
Hence if one is concerned with performance bounds for distributions on $Y$
that do not depend on the variable $X$, then the tightest bound is obtained
using $\Lambda _{c}^{1}$.

Before studying properties of the risk-sensitive measures, we note some
elementary generalizations that are possible. Suppose that the nominal
distribution of $X$ and $Y$ is not of product form. In the setting of the
first form, we could let $\gamma $ be the marginal distribution of $Y$ and
let $\mu (dx|y)$ be the conditional distribution of $X$ given $Y=y$. Using
the extended definition%
\begin{equation*}
\bar{\Lambda}_{c}^{1}=\frac{1}{c}\log \int_{\mathcal{Y}}e^{\int_{\mathcal{X}%
}cF(x,y)\mu (dx|y)}\gamma (dy),
\end{equation*}%
we obtain the bound%
\begin{equation*}
\int_{\mathcal{Y}}\int_{\mathcal{X}}F(x,y)\mu (dx|y)\theta (dy)\leq \frac{1}{%
c}R\left( \theta (dy)\left\Vert \gamma (dy)\right. \right) +\bar{\Lambda}%
_{c}^{1}.
\end{equation*}%
For the second form, we let $\mu $ be the marginal distribution of $X$ and
let $\gamma (dy|x)$ be the conditional distribution of $Y$ given $X=x$. With
the definition%
\begin{equation*}
\bar{\Lambda}_{c}^{2}=\frac{1}{c}\int_{\mathcal{X}}\left[ \log \int_{%
\mathcal{Y}}e^{cF(x,y)}\gamma (dy|x)\right] \mu (dx),
\end{equation*}%
we obtain%
\begin{equation*}
\int_{\mathcal{X}}\int_{\mathcal{Y}}F(x,y)\theta (dy|x)\mu (dx)\leq \frac{1}{%
c}\int_{\mathcal{X}}R\left( \theta (dy|x)\left\Vert \gamma (dy|x)\right.
\right) \mu (dx)+\bar{\Lambda}_{c}^{2}.
\end{equation*}

It is indeed sometimes useful to allow the distribution of one type of
variable to depend on the value taken by the other type of variable. As an
elementary example, consider the situation where an aleatoric variable
appears in a system, and while it is known that this variable has a Gaussian
distribution with mean zero, the variance of the variable is however not
known. Then the variance itself might be modeled as an uncertainty whose
distribution is not known precisely, i.e., as an epistemic uncertainty. In
this case $\bar{\Lambda}_{c}^{1}$ would be the relevant risk-sensitive
formulation. An example of this sort is presented in Subsection \ref%
{subsub:depend_uncertainties}.

\begin{remark}
\emph{In problems of regulatory assessment epistemic uncertainty can pose a
unique set of challenges. The collection of data and model validation are
often particularly difficult or expensive, while at the same time stringent
bounds on performance might be demanded. In these circumstances, the
framework presented here, wherein tight bounds are obtained for a known
family of alternative models, would seem very natural. Prior to the
collection of data or testing of samples to determine compliance, all
critical elements of the model, including selection of the nominal model,
size of the family of alternative models allowed by the relative entropy
bounds, and the performance measures that would be required to hold
uniformly for this family, would all be determined by negotiation among the
interested parties. This would allow various competing needs (e.g., expense
of model validation versus adequate protection of consumers) to be balanced
according to the interests of the participants.}
\end{remark}

\subsection{Properties of the risk-sensitive forms}

In this section we prove properties of the two hybrid risk-sensitive forms.
In particular we derive an inequality relating the two, and study the limit
as $c\rightarrow \infty $. For the results to hold as stated we often need $%
F $ be bounded from below. This is a mild assumption. Indeed, standard
measures of performance such as second moments and error probabilities
satisfy this condition.

The first proposition shows the inequality between the two hybrid forms.

\begin{proposition}
Assume that $F$ is bounded from below and consider the functionals defined
in (\ref{std_cost}), (\ref{rs_cost}), (\ref{hybrid1}) and (\ref{hybrid2}).
Then the inequalities (\ref{inequalities}) hold.
\end{proposition}

\begin{proof}
The only inequality which does not follow from Jensen's inequality is $%
\Lambda _{c}^{1}\leq \Lambda _{c}^{2}$. The relative entropy duality as
stated in (\ref{RE_duality}) applied to $\Lambda _{c}^{1}$ gives 
\begin{equation*}
\Lambda _{c}^{1}=\sup_{\theta \in \mathcal{P}(\mathcal{Y})}\left[ \int_{%
\mathcal{Y}}\int_{\mathcal{X}}F(x,y)\mu (dx)\theta (dy)-\frac{1}{c}R\left(
\theta (dy)\left\Vert \gamma (dy)\right. \right) \right] ,
\end{equation*}%
where $\mathcal{P}(\mathcal{Y})$ is the space of probability measures on $%
\mathcal{Y}$. The same representation applied to%
\begin{equation*}
\frac{1}{c}\log \int_{\mathcal{Y}}e^{cF(x,y)}\gamma (dy)
\end{equation*}%
gives 
\begin{equation*}
\frac{1}{c}\log \int_{\mathcal{Y}}e^{cF(x,y)}\gamma (dy)=\sup_{\theta \in 
\mathcal{P}(\mathcal{Y})}\left[ \int_{\mathcal{Y}}F(x,y)\theta (dy)-\frac{1}{%
c}R\left( \theta (dy)\left\Vert \gamma (dy)\right. \right) \right] .
\end{equation*}%
The supremum operation in the last display can be done in such a way that an
optimizing (or near optimizing) $\theta $ is a Borel measurable function of $%
x$, i.e., a stochastic kernel $\theta (dy|x)$ on $\mathcal{Y}$ given $%
\mathcal{X}$ \cite{dupell4}. Let $\mathcal{P}(\mathcal{Y}|\mathcal{X})$
denote the space of such stochastic kernels. Integrating the last display
with respect to $\mu $ gives 
\begin{equation*}
\Lambda _{c}^{2}=\sup_{\theta \in \mathcal{P}(\mathcal{Y}|\mathcal{X})}\left[
\int_{\mathcal{X}}\int_{\mathcal{Y}}F(x,y)\theta (dy|x)\mu (dx)-\frac{1}{c}%
\int_{\mathcal{X}}R\left( \theta (dy|x)\left\Vert \gamma (dy)\right. \right)
\mu (dx)\right] .
\end{equation*}%
Since $\mathcal{P}(\mathcal{Y})\subset \mathcal{P}(\mathcal{Y}|\mathcal{X})$%
, $\Lambda _{c}^{2}\geq \Lambda _{c}^{1}$ follows.\medskip
\end{proof}

The next proposition shows how to obtain tight bounds on the performance
over a family of distributions defined in terms of a maximum relative
entropy $B$.

\begin{proposition}
\label{prop2}Consider the functionals defined in (\ref{rs_cost}), (\ref%
{hybrid1}) and (\ref{hybrid2}), and let $D=\left\{ c:\Lambda _{c}<\infty
\right\} $ (resp., $D^{i}=\left\{ c:\Lambda _{c}^{i}<\infty \right\} ,i=1,2$%
). Assume that the interior of $D$ (resp., $D^{i}$) is nonempty. Then $%
\Lambda _{c}$ (resp., $\Lambda _{c}^{i}$) is differentiable on the interior
of $D$ (resp., $D^{i}$). Assume also that $F$ is bounded from below by zero.
Then $c\rightarrow \Lambda _{c}$ (resp., $c\rightarrow \Lambda _{c}^{i}$) is
nondecreasing for $c\geq 0$. Let $B>0$ be given. Then there is a unique $%
c\in (0,\infty ]$ at which 
\begin{equation*}
c\rightarrow \frac{1}{c}B+\Lambda _{c}\quad \left( \text{resp., }\frac{1}{c}%
B+\Lambda _{c}^{i}\right)
\end{equation*}%
attains a local minimum, where the statement that the minimum occurs at $%
c=\infty $ means that $\Lambda _{c}+B/c>\Lambda _{\infty }$ for a well
defined limit $\Lambda _{\infty }$ and all $c<\infty $.
\end{proposition}

\begin{proof}
To simplify we give the proof only for the case of $\Lambda _{c}$ and omit
the $y$ variable. Thus we consider $\Lambda _{c}=\frac{1}{c}\log \int_{%
\mathcal{X}}e^{cF(x)}\mu (dx)$. Proofs for the other functionals are
analogous. It is well known that $H(c)=\log \int_{\mathcal{X}}e^{cF(x)}\mu
(dx)$ is a convex function taking values in $\mathbb{R}\cup \left\{ \infty
\right\} $, and strictly convex and infinitely differentiable on the
interior of $\left\{ c:H(c)<\infty \right\} $ (see, e.g., \cite{var}).

Next we prove the monotonicity. Differentiating gives%
\begin{equation*}
H^{^{\prime }}(c)=\frac{\int_{\mathcal{X}}F(x)e^{cF(x)}\mu (dx)}{\int_{%
\mathcal{X}}e^{cF(x)}\mu (dx)}.
\end{equation*}%
Since $F\geq 0$ the derivative is non-negative, and convexity implies it is
non-decreasing. Using%
\begin{equation*}
\frac{1}{c}H(c)=\frac{1}{c}\left( H(c)-H(0)\right) =\frac{1}{c}%
\int_{0}^{c}H^{^{\prime }}(s)ds,
\end{equation*}%
it follows that $\Lambda _{c}=H(c)/c$ is nondecreasing for $c\geq 0$.

First assume that $M\doteq \sup D=\infty $, and observe that 
\begin{equation*}
\frac{d}{dc}\left[ \frac{1}{c}B+\frac{1}{c}H(c)\right] =\frac{1}{c^{2}}\left[
-B+cH^{^{\prime }}(c)-H(c)\right] .
\end{equation*}%
Since $H$ is strictly convex and $H^{^{\prime }}(0)\geq 0$, the mapping $%
f(c)\rightarrow cH^{^{\prime }}(c)-H(c)$ is monotone increasing and takes
the value $0$ at $c=0$ (see Figure 1).


\begin{figure}[h!]
\begin{center}
\includegraphics[width=2.64in,height=2.67in]{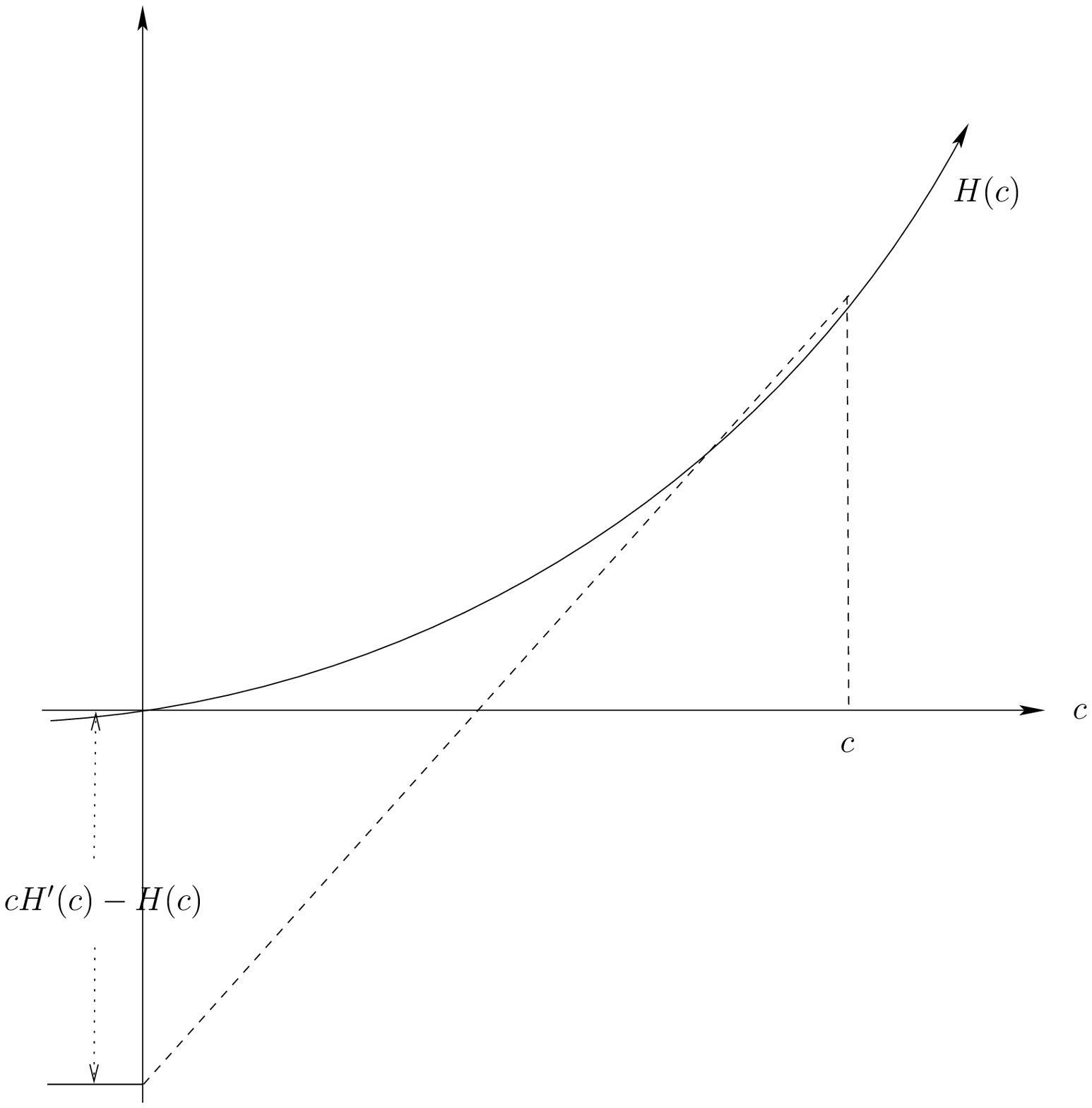}
\end{center}
\caption{$f(c)$ is monotone increasing}
\label{Monotone increasing}
\end{figure}

Let $K$ denote the limit of $f(c)$ as $c\rightarrow \infty $. If $K=\infty $
then there is a unique solution to 
\begin{equation*}
cH^{^{\prime }}(c)-H(c)=B,
\end{equation*}%
and we are done. The same is true if $0\leq B<K$. Hence the only case left
is when $B\geq K$. In this case $(B+H(c))/c$ is monotone decreasing for all $%
c\in \lbrack 0,\infty )$. Since $H(c)/c\geq H^{^{\prime }}(0)\geq 0$, there
is a well-defined limit $\Lambda _{\infty }$ that is necessarily the minimum.

Next assume that $M\doteq \sup D\in (0,\infty )$. We claim that in this case 
$cH^{^{\prime }}(c)-H(c)\uparrow \infty $ as $c\uparrow M$, and therefore we
can argue just as before. By monotone convergence it must be that $%
H(c)\uparrow \infty $ as $c\uparrow M$, which implies that necessarily $%
H^{^{\prime }}(c)\uparrow \infty $ as $c\uparrow M$. To prove the claim, let 
$0<\beta <c<M$. Then since $H^{^{\prime }}(s)$ is increasing%
\begin{equation*}
H(c)=\int_{0}^{c}H^{^{\prime }}(s)ds\leq \beta H^{^{\prime }}(\beta
)+(c-\beta )H^{^{\prime }}(c).
\end{equation*}%
This implies 
\begin{equation*}
cH^{^{\prime }}(c)-H(c)\geq \beta \left[ H^{^{\prime }}(c)-H^{^{\prime
}}(\beta )\right] .
\end{equation*}%
Letting $c\uparrow M$ and using that $\beta >0$, the claim is proved, thus
completing the proof of the theorem.
\end{proof}

\medskip

One of the situations of interest when aleatoric and epistemic variables
appear simultaneously is the case where all that is known regarding the
epistemic variables are bounds. It turns out that this problem is well posed
(i.e., the relevant performance criteria are finite) essentially in those
cases where $\Lambda _{\infty }^{1}<\infty $. In fact, when this is the case
the value of $\Lambda _{\infty }^{1}$ can be used to establish optimal
bounds subject to the constraint on the epistemic variables. For simplicity,
we assume that the set $A$ appearing in the statement of the following
theorem is bounded and that $\gamma $, the nominal distribution for the
epistemic variables, is the uniform distribution (in the limit $c\rightarrow
\infty $ the precise form of the nominal distribution is not important, and
in fact it is only the support of the distribution that matters in the
limit). If $A$ is not bounded one can use an appropriate distribution whose
support is all of $A$ (e.g., the exponential distribution could be used for $%
A=[0,\infty )$. The choice of $\gamma $ as uniform when $A$ is bounded is
simplest and also one that is convenient for most uses. A related statement
holds for $\Lambda _{\infty }^{2}$, but it is not as useful (at least in
this setting), since $\Lambda _{\infty }^{1}\leq \Lambda _{\infty }^{2}$.

\begin{theorem}
\label{thm:c_to_infty}Suppose that $A\subset \mathcal{Y}$ is bounded and the
closure of its interior and that $\gamma $ is the uniform distribution on $A$%
. Assume that $F$ is lower semicontinuous in $y$ for each $x\in \mathcal{X}$
and bounded from below. Define the risk-sensitive functional $\Lambda
_{c}^{1}$ by (\ref{hybrid1}) and let $\Lambda _{\infty
}^{1}=\lim_{c\rightarrow \infty }\Lambda _{c}^{1}$. Then 
\begin{equation*}
\sup_{y\in A}\int_{\mathcal{X}}F(x,y)\mu (dx)=\Lambda _{\infty }^{1}.
\end{equation*}
\end{theorem}

\begin{proof}
Since $c\Lambda _{c}^{1}$ is convex the limit of $\Lambda _{c}^{1}$ is well
defined, though it might take the value $\infty $. We give the proof for the
case $\Lambda _{\infty }^{1}<\infty $. The extension to the case $\Lambda
_{\infty }^{1}=\infty $ is straightforward.

We first prove that the left side of the last display is bounded above by
the right side. Fix any $\bar{y}\in A^{\circ }$, the interior of $A$. For
small $\varepsilon >0$ let $B_{\varepsilon }$ be the ball in $A$ about $\bar{%
y}$ of volume $\varepsilon $, and let $M$ be the volume of $A$. Let $\theta
(dy)$ be the measure whose density with respect to Lebesgue measure is $%
1/\varepsilon $ on $B_{\varepsilon }$, and zero elsewhere. Then 
\begin{eqnarray*}
R\left( \theta (dy)\left\Vert \gamma (dy)\right. \right)
&=&\int_{B_{\varepsilon }}\log \left[ (1/\varepsilon )/(1/M)\right] \theta
(dy) \\
&=&\left( 1/\varepsilon \right) \log \left[ (1/\varepsilon )/(1/M)\right] 
\text{.}
\end{eqnarray*}%
Then (\ref{first}) gives%
\begin{equation}
\int_{\mathcal{X}}\int_{B_{\varepsilon }}F(x,y)\theta (dy)\mu (dx)\leq \frac{%
1}{c}\left( 1/\varepsilon \right) \log \left[ (1/\varepsilon )/(1/M)\right]
+\Lambda _{c}^{1}.  \label{REbound}
\end{equation}%
Since $y\rightarrow F(x,y)$ is lower semicontinuous 
\begin{equation*}
\liminf_{\varepsilon \rightarrow 0}\int_{B_{\varepsilon }}F(x,y)\theta
(dy)\geq F(x,\bar{y}).
\end{equation*}%
Letting first $c\rightarrow \infty $ and then $\varepsilon \rightarrow 0$ in
(\ref{REbound}), Fatou's lemma gives $\int_{\mathcal{X}}F(x,\bar{y})\mu
(dx)\leq \Lambda _{\infty }^{1}$. Since $\bar{y}\in A^{\circ }$ is
arbitrary, the lower semicontinuity and another use of Fatou give the bound
for all $y\in A$.

For the reverse inequality, we use the fact that the minimizing measure in
the definition of $\Lambda _{c}^{1}$ exists. In fact, 
\begin{equation*}
\int_{\mathcal{Y}}\int_{\mathcal{X}}F(x,y)\mu (dx)\theta _{c}^{\ast }(dy)=%
\frac{1}{c}R\left( \theta _{c}^{\ast }(dy)\left\Vert \gamma (dy)\right.
\right) +\Lambda _{c}^{1}
\end{equation*}%
precisely at $\theta ^{\ast }$ defined by%
\begin{equation*}
\frac{d\theta _{c}^{\ast }(\cdot )}{d\gamma (\cdot )}(y)=e^{\int_{\mathcal{X}%
}cF(x,y)\mu (dx)}\left/ \int_{A}e^{\int_{\mathcal{X}}cF(x,y)\mu (dx)}\gamma
(dy)\right.
\end{equation*}%
(see \cite[Proposition 1.4.2]{dupell4}). By the non-negativity of relative
entropy%
\begin{eqnarray*}
\Lambda _{\infty }^{1} &=&\lim_{c\rightarrow \infty }\Lambda _{c}^{1} \\
&\leq &\limsup_{c\rightarrow \infty }\int_{\mathcal{Y}}\int_{\mathcal{X}%
}F(x,y)\mu (dx)\theta _{c}^{\ast }(dy) \\
&\leq &\sup_{y\in A}\int_{\mathcal{X}}F(x,y)\mu (dx).
\end{eqnarray*}%
\qquad
\end{proof}

\section{Examples and computational methods}

To indicate how the robust performance bounds might be used in practice, we
present some numerical examples, which illustrate the relationship between
the different risk sensitive integrals and the relevant techniques and
computational issues. We first review the polynomial chaos techniques that
are used to compute the risk-sensitive integrals. The reader familiar with
this material can skip to the next section.

\begin{remark}
\emph{Although previously the random variables with known and unknown
distributions were denoted by }$X$\emph{\ and }$Y$\emph{, in the rest of the
paper these random variables will be denoted by }$Z_{1}$\emph{\ and }$Z_{2}$%
\emph{. This is done to free up }$x$\emph{\ and }$y$\emph{\ to be spatial
variables in various PDE and related equations that might be used to define
the mapping }$F$\emph{.}
\end{remark}

\subsection{Review of polynomial chaos methods}

This subsection reviews the polynomial chaos method that we use to compute
the risk-sensitive integrals. The reader familiar with this material can
skip to Section 4.2.

To calculate risk sensitive integrals, one could use Monte Carlo
integration, which requires evaluating $F(Z_{1},Z_{2})$ for many replicas of 
$(Z_{1},Z_{2})$. This can get costly (though not necessarily for the given
examples), and so we utilize generalized polynomial chaos (gPC) methods to
approximate $F(z_{1},z_{2})$ and evaluate various statistical properties of $%
F(Z_{1},Z_{2})$ (see, for example, \cite{xiukar2}). We briefly recall the
mathematical framework of polynomial chaos methods, along with a simple
example in one dimension.

Let $(\Omega ,\mathcal{A}$\thinspace $,\mathcal{P)}$ be a probability space
where $\Omega $ is the sample space, $\mathcal{A}$ the associated $\sigma $%
-algebra, and $\mathcal{P}$ the probability measure on $\mathcal{A}$. Define
a random vector $\mathbf{Z}(\omega )\doteq (Z_{1}(\omega ),\ldots
,Z_{N}(\omega ))\in \mathbb{R}^{N}$ on this probability space. Consider a $d$%
-dimensional bounded domain $D\subset \mathbb{R}^{d}$ with boundary $%
\partial D$ and let $t\in \lbrack 0,T],T\in (0,\infty )$. We consider random
fields $u(t,x;\mathbf{Z(}\omega )):\bar{D}\times \lbrack 0,T]\times \Omega
\rightarrow \mathbb{R}$ that are defined by requiring that for a.e. $\omega
\in \Omega $%
\begin{equation*}
\mathcal{L}(t,x,u;\mathbf{Z(}\omega ))=f(t,x;\mathbf{Z(}\omega )),\qquad
(x,t,\omega )\in D\times \lbrack 0,T]\times \Omega ,
\end{equation*}%
subject to the boundary and initial conditions%
\begin{eqnarray*}
\mathcal{B}(t,x,u;\mathbf{Z(}\omega )) &=&g(t,x;\mathbf{Z(}\omega )),\qquad
(x,t,\omega )\in \partial D\times \lbrack 0,T]\times \Omega , \\
u(t,x;\mathbf{Z(}\omega )) &=&u_{0}(x;\mathbf{Z}(\omega )),\qquad
(x,t,\omega )\in D\times \{0\}\times \Omega .
\end{eqnarray*}%
Here $x=(x_{1},...,x_{d})\in \mathbb{R}^{d},\ \mathcal{L}$ is a linear or
nonlinear differential operator, and $\mathcal{B}$ is a boundary operator.
We assume for simplicity that a unique classical sense solution exists to
the differential equation and/or boundary conditions. Note that for the
first two examples of the last subsection the problem involves only time
dependence, and hence there is neither a spatial variable nor a boundary
condition.

We assume that $\left\{ Z_{i}\right\} _{i=1}^{N}$ are independent random
variables, with support $\left\{ \Gamma _{i}\right\} _{i=1}^{N}$ and with
probability density functions $\left\{ \rho _{i}(z_{i})\right\} _{i=1}^{N}$,
respectively. Hence the joint density is $\rho (\mathbf{z)=}\Pi
_{i=1}^{N}\rho _{i}(z_{i})$. Let $\Omega $ be the canonical space $\times
_{i=1}^{N}\Gamma _{i}$, in which case $\mathbf{Z(}\omega )=\omega $ and we
identify $\omega $ with $\mathbf{z}$. Thus the equations above become%
\begin{equation}
\mathcal{L}(t,x,u;\mathbf{z})=f(t,x;\mathbf{z}),\qquad (x,t,\mathbf{z})\in
D\times \lbrack 0,T]\times \Omega ,  \label{generalDE}
\end{equation}%
subject to the boundary and initial conditions%
\begin{eqnarray}
\mathcal{B}(t,x,u;\mathbf{z}) &=&g(t,x;\mathbf{z}),\qquad (x,t,\mathbf{z)}%
\in \partial D\times \lbrack 0,T]\times \Omega ,  \label{generalBC} \\
u(t,x;\mathbf{z}) &=&u_{0}(x;\mathbf{z}),\qquad (x,t,\mathbf{z)}\in D\times
\{0\}\times \Omega .  \notag
\end{eqnarray}%
We consider approximating the mapping as a function of the stochastic
variable, i.e., $\mathbf{z}\rightarrow u(t,x;\mathbf{z})$ at some particular 
$(t,x)$, via a finite sum of orthogonal basis functions.

We first define finite dimensional subspaces for $L^{2}(\Gamma _{i})$
according to 
\begin{equation*}
W_{i}^{P_{i}}=\left\{ v:\Gamma _{i}\rightarrow \mathbb{R}:v\in \text{span}%
\left\{ \phi _{i,m}(z_{i})\right\} _{m=0}^{P_{i}}\right\} ,\qquad i=1,...,N.
\end{equation*}%
Here $P_{i}$ represents the highest degree of the polynomial basis function,
and $\left\{ \phi _{i,m}(z_{i})\right\} $ are a set of orthonormal
polynomials with respect to the weight $\rho _{i}$, i.e., for $m\neq n$%
\begin{equation*}
\int_{\Gamma _{i}}\phi _{i,m}(z_{i})\phi _{i,n}(z_{i})\rho
_{i}(z_{i})dz_{i}=0
\end{equation*}%
and 
\begin{equation*}
\int_{\Gamma _{i}}\phi _{i,n}^{2}(z_{i})\rho _{i}(z_{i})dz_{i}=1.
\end{equation*}%
The orthonormal basis representation is determined by the probability
density function $\rho _{i}$. For example, if the density is uniform or
Gaussian, then Legendre or Hermite orthogonal polynomials, respectively, are
used. A table of polynomial basis functions and their respective
distributions is listed at the end of this section (also see \cite{xiu1}). A
finite dimensional subspace for $L^{2}(\Gamma )$, where $\Gamma \doteq
\Gamma _{1}\times \cdots \times \Gamma _{N}=\Omega $, can either be defined
as 
\begin{equation*}
W_{N}^{P}=\bigotimes\limits_{|\mathbf{P}|\leq P}W_{i}^{P_{i}}
\end{equation*}%
where the tensor product is over all combinations of the multi-index $%
\mathbf{P}=(P_{1},\dots ,P_{N})\in \mathbb{N}_{0}^{N}$ with $\left\vert 
\mathbf{P}\right\vert =\sum_{i=1}^{N}P_{i}\leq P$, or 
\begin{equation*}
\tilde{W}_{N}^{P}=\bigotimes\limits_{i=1}^{N}W_{i}^{P}.
\end{equation*}%
Thus, $W_{N}^{P}$ is the space of $N$-variate orthogonal polynomials of
total degree at most $P$, whereas $\tilde{W}_{N}^{P}$ is the full tensor
product of the one-dimensional polynomial spaces with each highest degree $P$%
. Note that $\dim (W_{N}^{P})=\binom{N+P}{P}$ and $\dim (\tilde{W}%
_{N}^{P})=(P+1)^{N}$, and that for large $N$, $\binom{N+P}{P}\ll (P+1)^{N}$.
Since our examples only consider $N=2$, we will use the full tensor product
space, $\tilde{W}_{N}^{P}$. Let $\Phi _{j}(\mathbf{z)},j=1,\ldots ,M$ denote
the elements of $\tilde{W}_{N}^{P}$, where $M=\dim (\tilde{W}_{N}^{P})$.

There are two standard methods for constructing gPC approximations, referred
to as the stochastic Galerkin method (see, e.g., \cite{xiu1}) and the
stochastic collocation method (see, e.g., \cite{xiuhes}). Since the
calculations presented below use only the collocation method, we restrict
discussion to this method. With stochastic collocation, we first consider an
approximation to the solution $u(t,x;\mathbf{z)}$ in terms of a Lagrange
interpolant of the form 
\begin{equation}
v(t,x;\mathbf{z)}=\sum_{k=1}^{N_{p}}v(t,x;\mathbf{z}_{k})\mathbf{L}_{k}(%
\mathbf{z}),  \label{lagrange_form}
\end{equation}%
where $\mathbf{L}_{k}(\mathbf{z})$ is a Lagrange polynomial of degree $\dim
(W_{N}^{P})$, satisfies $\mathbf{L}_{k}(\mathbf{z}_{l})=\delta _{kl}$, and $%
\left\{ \mathbf{z}_{k}\right\} _{k=1}^{N_{p}}$ are a set of prescribed nodes
in the $N$-dimensional space $\Gamma $. We require that the residuals 
\begin{eqnarray*}
R(t,x;\mathbf{z)} &\doteq &\mathcal{L}(t,x,v;\mathbf{z})-f(t,x;\mathbf{z)} \\
R^{BC}(t,x;\mathbf{z)} &\doteq &\mathcal{B}(t,x,v;\mathbf{z})-g(t,x;\mathbf{z%
}) \\
R^{IC}(x;\mathbf{z)} &\doteq &v(0,x;\mathbf{z})-u_{0}(x;\mathbf{z})
\end{eqnarray*}%
vanish at the collocation points $\left\{ \mathbf{z}_{k}\right\}
_{k=1}^{N_{p}}$, i.e., 
\begin{eqnarray*}
\mathcal{L}(t,x,v(t,x;\mathbf{z}_{k});\mathbf{z}_{k})-f(t,x;\mathbf{z}_{k}%
\mathbf{)} &=&0,\text{ for }k=1,\dots ,N_{p} \\
\mathcal{B}(t,x,v(t,x;\mathbf{z}_{k});\mathbf{z}_{k})-g(t,x;\mathbf{z}_{k})
&=&0,\ \text{for }k=1,\dots ,N_{p} \\
v(0,x;\mathbf{z}_{k})-u_{0}(x;\mathbf{z}_{k}) &=&0,\ \text{for }k=1,\dots
,N_{p}.
\end{eqnarray*}%
Thus, the stochastic collocation method produces a set of $N_{p}$ decoupled
equations, where each value of $v(t,x;\mathbf{z}_{k})$ coincides with the
exact solution $u(t,x;\mathbf{z})$ for the given $\mathbf{z}_{k}$, since
both satisfy the same equations.

Once the values of $\left\{ v(t,x;\mathbf{z}_{k})\right\} _{k=1}^{N_{p}}$
have been determined at the collocation points, one can construct a gPC
approximation in the orthogonal basis representation in the form%
\begin{equation}
v(t,x;\mathbf{z)}=\sum_{j=1}^{M}\hat{v}_{j}(t,x)\Phi _{j}(\mathbf{z}),
\label{orthog_rep}
\end{equation}%
which, under certain conditions, will be equivalent to the Lagrange basis
formulation. The coefficients $\left\{ \hat{v}_{j}\right\} _{j=1}^{M}$ can
be determined by inverting a Vandermonde matrix, where invertibility is
dependent on the choice of collocation points. If one chooses quadrature or
cubature\footnote{%
cubature points just refers to quadratures in more than one dimension.}
points for the collocation points, then invertibility is guaranteed.
Furthermore, if the cubature rule is exact up to polynomials of degree $2M$,
inversion of the Vandermonde matrix is not necessary\footnote{%
This is especially useful if the Vandermonde matrix is ill-conditioned.},
and the coefficients $\left\{ \hat{v}_{j}\right\} _{j=1}^{M}$ can be
computed by evaluating 
\begin{equation}
\hat{v}_{j}(t,x)=\sum_{k=1}^{N_{p}}v(t,x;\mathbf{z}_{k})\Phi _{j}(\mathbf{z}%
_{k})w_{k},  \label{coeff_orthog}
\end{equation}%
where $\left\{ w_{k}\right\} _{k=1}^{N_{p}}$ are the respective weights
according to the choice of quadrature or cubature points.

Statistical information is easy to obtain from the form (\ref{orthog_rep}).
For example,%
\begin{equation*}
\mathbb{E}_{\mathbf{z}}\mathbb{[}v(t,x;\mathbf{z})]=\int_{\Gamma }\left(
\sum_{j=1}^{M}\hat{v}_{j}(t,x)\Phi _{j}(\mathbf{z})\right) \rho (\mathbf{z})d%
\mathbf{z}=\hat{v}_{1}(t,x)
\end{equation*}%
and%
\begin{equation*}
\mathbb{E}_{\mathbf{z}}\mathbb{[(}v(t,x;\mathbf{z}))^{2}]=\int_{\Gamma
}\left( \sum_{j=1}^{M}\hat{v}_{j}(t,x)\Phi _{j}(\mathbf{z})\right) ^{2}\rho (%
\mathbf{z})d\mathbf{z}=\sum_{j=1}^{M}\hat{v}_{j}(t,x)^{2}
\end{equation*}%
as a consequence of the orthonormality. Other statistical information, such
as higher moments and sensitivity coefficients, are also easily calculated
(see \cite{xiu1}). The ease with which one can calculate such moments
explains why orthogonal representation (\ref{orthog_rep}) is preferred over
the Lagrange form (\ref{lagrange_form}).

To illustrate the approach, we look at a simplified version of Example \ref%
{example1}. Let $g(z_{2})=u_{0}$, so that the only uncertainty is in the
decay rate $k$, which depends on the random variable $Z$. The simplified
problem becomes

\begin{equation*}
\frac{d}{dt}u(t)=-k(z)u(t),\qquad u(0)=u_{0},
\end{equation*}%
where $u(t;z)$ denotes the solution parameterized by $z$, with $z$ a point
in the range of $Z$. Requiring that the residuals vanish at the collocation
points $\left\{ z_{n}\right\} _{n=1}^{N_{p}}$ gives%
\begin{equation*}
R(t;z_{n})=\frac{d}{dt}v(t;z_{n})+k(z_{n})v(t;z_{n})=0
\end{equation*}%
and%
\begin{equation*}
R^{IC}(z_{n})=v(0;z_{n})-u_{0}=0.
\end{equation*}%
Note that there are indeed $N_{p}$ decoupled equations, which are solved
independently for each $n$. Here the solutions are exact:%
\begin{equation*}
v(t;z_{n})=u_{0}e^{-k(z_{n})t},\ for=1,\dots ,N_{p}.
\end{equation*}%
One can then obtain the approximation (\ref{orthog_rep}) with respect to the
orthogonal basis, assuming that the appropriate collocation points and
weights have been chosen.

We illustrate this example with $k(z)=z$ and $Z\sim \mathcal{N}(\mu
=0,\sigma ^{2}=1)$. Here the $\left\{ \Phi _{j}\right\} _{j=1}^{M}$ are the
Hermite polynomials, which are orthonormal with respect to the weighting
function $\rho (z)=e^{-(z-\mu )^{2}/2\sigma ^{2}}/\sigma \sqrt{2\pi }$ with
support $\Gamma =(-\infty ,+\infty )$. Figure \ref{fig2} shows the
stochastic collocation approximation $v(1,z)$ at $t=1$ as a function of $z$,
with $N_{p}=2,3,4,5$, versus the exact solution $u_{exact}(1,z)=e^{-z}$.
Note that the collocation approximation is an interpolation of the exact
solution at $N_{p}$ points.


\begin{figure}[h!]
\begin{center}
\includegraphics[width=5in,height=3.45in]{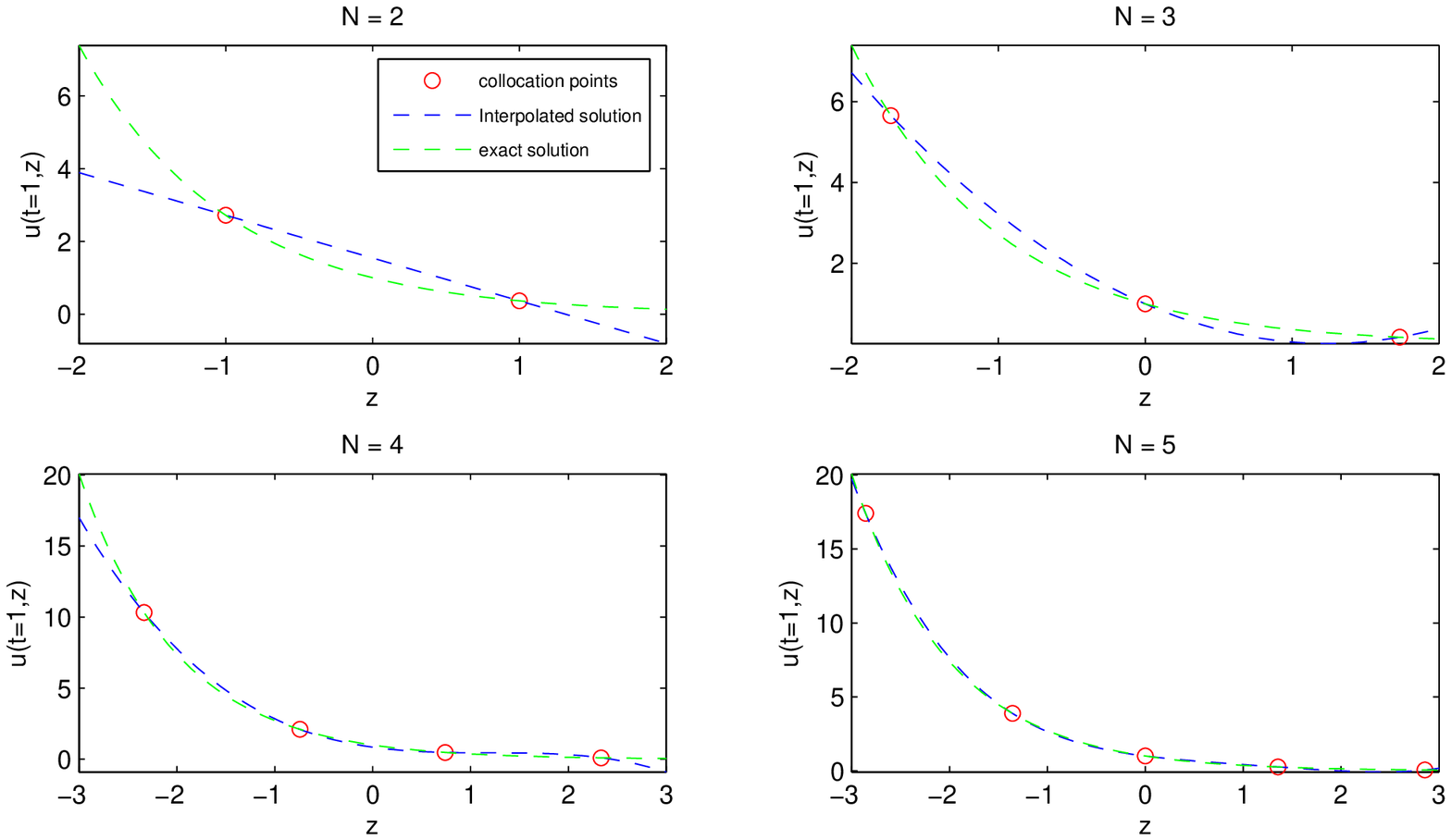}
\end{center}
\caption{Interpolated solutions versus the exact solution at $t=1$.}
\label{fig2}
\end{figure}


\begin{figure}[h!]
\begin{center}
\includegraphics[width=3.79in,height=2.85in]{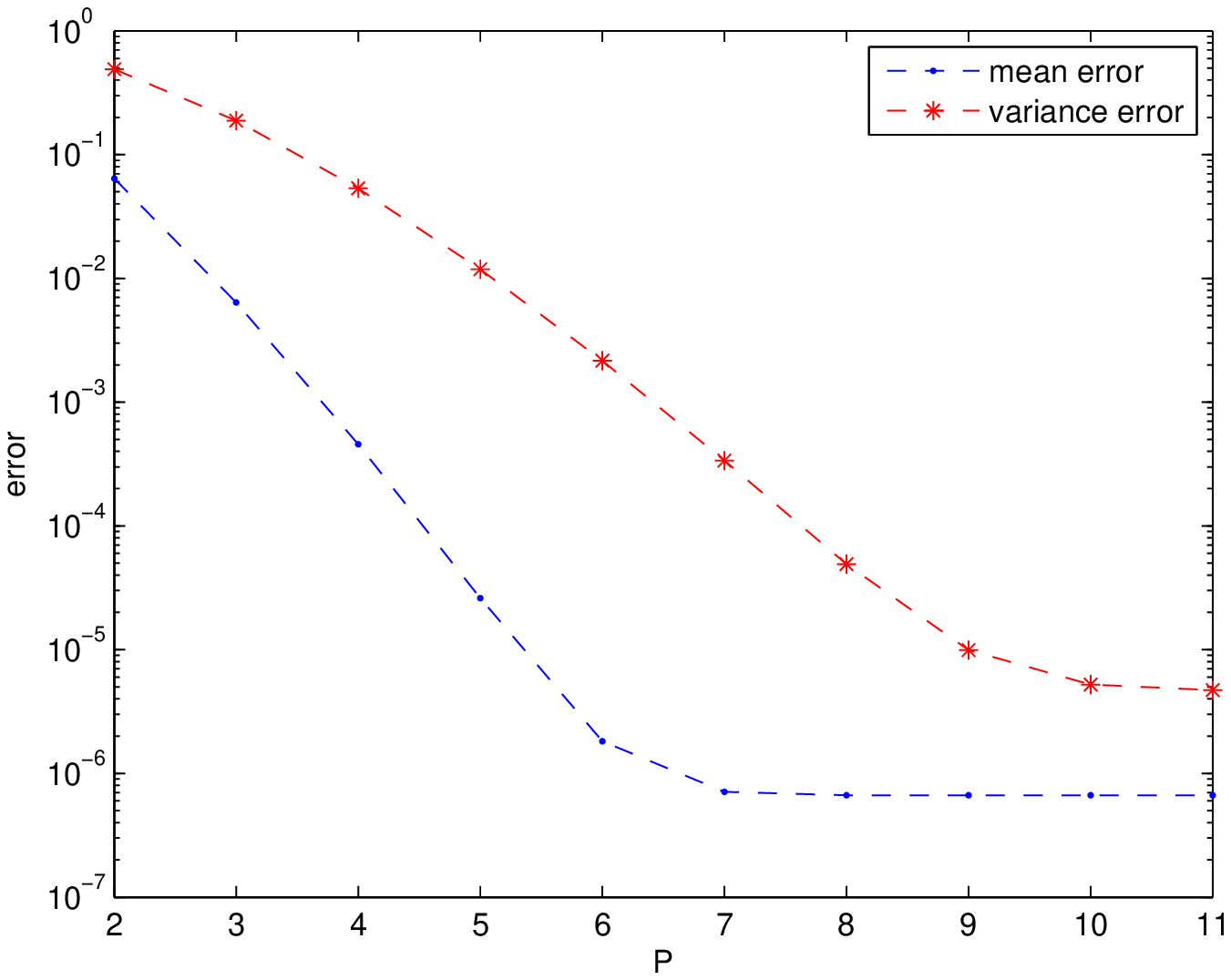}
\end{center}
\caption{Relative errors.}
\label{fig3}
\end{figure}

Figure \ref{fig3} shows the relative error between the exact and
approximated mean and variance in a semilog plot. The linear decrease in
relative error on the semilog plot implies an exponential decay in the
error. Errors become a constant at the limits of the accuracy of the
numerical scheme.

\subsection{Numerical examples and interpretation}

\bigskip

\setcounter{figure}{3} We now discuss numerical calculation of the
risk-sensitive integrals. For convenience, we recall the three types of
risk-sensitive integrals introduced in Section \ref{Section:RSforms}, the
ordinary risk sensitive cost $\Lambda _{c}$ and the two hybrid integrals $%
\Lambda _{c}^{1}$ and $\Lambda _{c}^{2}$:%
\begin{eqnarray}
\Lambda _{c} &=&\frac{1}{c}\log \int_{\mathcal{Z}_{1}}\int_{\mathcal{Z}%
_{2}}e^{cF(z_{1},z_{2})}\gamma (dz_{2})\mu (dz_{1}),\   \label{rs_integrals}
\\
\ \Lambda _{c}^{1} &=&\frac{1}{c}\log \int_{\mathcal{Z}_{2}}e^{\int_{%
\mathcal{Z}_{1}}cF(z_{1},z_{2})\mu (dz_{1})}\gamma (dz_{2}),  \notag \\
\ \ \Lambda _{c}^{2} &=&\frac{1}{c}\int_{\mathcal{Z}_{1}}\left[ \log \int_{%
\mathcal{Z}_{2}}e^{cF(z_{1},z_{2})}\gamma (dz_{2})\right] \mu (dz_{1}), 
\notag
\end{eqnarray}%
where $\mathcal{Z}_{1},\mathcal{Z}_{2}$ are the range spaces for the random
variables $Z_{1},Z_{2}$, respectively. Recall the relationships $\Lambda
_{c}^{1}\leq \Lambda _{c}^{2}\leq \Lambda _{c}$. It is assumed that the
distribution of $Z_{1}$ is known, but this is not true for $Z_{2}$.

In the stochastic collocation approach, we replace $F(z_{1},z_{2})$ in (\ref%
{rs_integrals}) with a polynomial interpolant, $\tilde{F}(z_{1},z_{2})$,
determined by the full tensor product of the one dimensional collocation
points, $\left\{ z_{1,i}\right\} _{i=1}^{N_{1,q}}\otimes \left\{
z_{2,i}\right\} _{i=1}^{N_{2,q}}$ in $\mathcal{Z}_{1}\times \mathcal{Z}_{2}$%
. Thus 
\begin{equation*}
\tilde{F}(z_{1},z_{2}\mathbf{)}=\sum_{j=1}^{M}\hat{F}_{j}(t,x)\Phi
_{j}(z_{1},z_{2}),
\end{equation*}%
where $M=N_{1,q}\cdot N_{2,q}\,$, $\Phi _{j}$ is the full tensor product
basis, and $\hat{F}_{j}(t,x)$ is computed using (\ref{coeff_orthog}). From
here, one can efficiently compute the risk sensitive integrals (\ref%
{rs_integrals}) via either Monte Carlo sampling or quadrature. For both
methods, once the samples or quadrature points have been chosen, one can
calculate the risk sensitive integrals for different values of $c$ without
resampling or choosing different quadrature points. Hence, it is
computationally inexpensive to compute the risk sensitive integral for
different values of $c$. Note that for higher dimensions, a sparse grid is
preferred and, in most cases, necessary for these types of computations. It
is typical to use the Smolyak algorithm to generate these sparse grids,
which are based on a one dimensional quadrature rule (for example, see
Section 4.1.2 in \cite{xiu2} for a more detailed explanation and further
references).

In this paper, we opt for the numerical quadrature approach. Since $\tilde{F}%
(z_{1},z_{2})$ can be evaluated very quickly, we can compute the risk
sensitive integrals with high accuracy and efficiency using a high order
quadrature rule. The number of quadrature points chosen to compute the risk
sensitive integrals need not be the same as the number of collocation points
used in computing the coefficients of the gPC expansion, though the type of
quadrature points is the same. For example, if we were to use Legendre-Gauss
quadrature points to evaluate the gPC coefficients, we would again use
Legendre-Gauss quadrature points to calculate the risk sensitive integrals.
In the examples presented below a much larger number of quadrature points
are used to compute the risk sensitive integrals. To avoid confusion, to the
set of quadrature points used to evaluate the risk sensitive integrals is
denoted by $\left\{ \tilde{z}_{1,i}\right\} _{i=1}^{\tilde{N}_{1,q}}\otimes
\left\{ \tilde{z}_{2,i}\right\} _{i=1}^{\tilde{N}_{2,q}}$, i.e., the full
tensor product of one dimensional quadrature points. In the examples, $%
\tilde{N}_{1,q}$ and$\ \tilde{N}_{2,q}$ equal $2^{8}$, while $N_{1,q}$ and$\
N_{2,q}$ are $8$ for smooth $F$ and $12$ when $F$ is an indicator functions.

We also implemented the Monte Carlo approach, but observed that numerical
quadrature gave a much better approximation for the same computational
effort, at least for our examples. It should also be noted that a variation
on traditional Monte Carlo is needed for the hybrid forms of the risk
sensitive integrals. For the original risk-sensitive cost function (\ref%
{rs_cost}), the standard approach is to use many independent samples of the
pair $(Z_{1},Z_{2})$. This produces an estimate which has a mean square
error of $\mathcal{O(}N^{-1/2})$, where $N$ is the total number of Monte
Carlo samples. However, for the first hybrid form (\ref{hybrid1}), one must
approximate the inner integral $\int_{\mathcal{X}}cF(z_{1},z_{2})\mu
(dz_{1}) $ for every fixed sample of $Z_{2}$. Thus one has to simulate more $%
Z_{1}$ samples than $Z_{2}$ samples. Similarly, to compute the second form
of the hybrid (\ref{hybrid2}), one must approximate the inner integral $%
\int_{\mathcal{Y}}e^{cF(z_{1},z_{2})}\gamma (dz_{2})$ for every fixed $Z_{1}$%
. In particular, for both hybrid forms one does not simply choose
independent samples of $(Z_{1},Z_{2})$.

\subsubsection{Independent aleatoric and epistemic uncertainties}

When the aleatoric and nominal epistemic variables are independent, the
numerical quadrature approach would evaluate the risk sensitive integrals
via the formulas

\begin{eqnarray*}
\Lambda _{c} &\approx &\frac{1}{c}\log \left( \sum_{i=1}^{\tilde{N}%
_{1,q}}\sum_{j=1}^{\tilde{N}_{2,q}}e^{cF(\tilde{z}_{1,i},\ \tilde{z}_{2,j})}%
\tilde{w}_{1,i}\tilde{w}_{2,j}\right) \  \\
\ \Lambda _{c}^{1} &\approx &\frac{1}{c}\log \left( \sum_{j=1}^{\tilde{N}%
_{2,q}}e^{\left( \sum_{i=1}^{\tilde{N}_{1,q}}cF(\tilde{z}_{1,i},\ \tilde{z}%
_{2,j})\tilde{w}_{1,i}\right) }\tilde{w}_{2,j}\right) \\
\ \ \Lambda _{c}^{2} &\approx &\frac{1}{c}\sum_{i=1}^{\tilde{N}_{1,q}}\left(
\log \sum_{j=1}^{\tilde{N}_{2,q}}e^{cF(\tilde{z}_{1,i},\tilde{z}_{2,j})}%
\tilde{w}_{2,j}\right) \tilde{w}_{1,i}.
\end{eqnarray*}%
where $\left\{ \tilde{z}_{1,i},\tilde{w}_{1,i}\right\} _{i=1}^{\tilde{N}%
_{1,q}},\left\{ \tilde{z}_{2,i},\tilde{w}_{2,i}\right\} _{i=1}^{\tilde{N}%
_{2,q}}$ are the pairs of quadrature points and weights chosen based on the
underlying distribution.

The first example is taken from \cite{xiueld} and involves a simple ODE. The
example is convenient because analytic solutions to the risk-sensitive
integrals can be obtained, and then compared with the numerical
approximations described in the next section. The second example is a random
oscillator, where the stochastic parameters are the spring and dampening
coefficients. The third example involves a heat equation with random heat
capacity and thermal conductivity.

\begin{example}
\label{example1}\emph{We consider }$F(Z_{1},Z_{2})$\emph{\ defined in terms
of the solution of the ODE}%
\begin{equation*}
\frac{d}{dt}u(t)=-k(z_{1})u(t),\quad u(0)=g(z_{2}),
\end{equation*}%
\emph{where }$k$\emph{\ and }$g$\emph{\ are known functions. Letting }$%
u(t;z_{1},z_{2})$\emph{\ denote the solution parameterized by }$%
(z_{1},z_{2}) $\emph{, set }$F(Z_{1},Z_{2})=h(u(t;Z_{1},Z_{2}))$\emph{,
where typical choices of }$h$ \emph{are }$h(u)=u,h(u)=u^{2},h(u)=\boldsymbol{%
1}\{a\leq u\leq b\}$\emph{.}
\end{example}

Here we take $k(z_{1})=z_{1},g(z_{2})=z_{2},$ let $Z_{1}\sim U[0,1]$, $%
Z_{2}\sim U[0,1]$, and compute risk sensitive integrals for $%
F(z_{1},z_{2})=[u(1;z_{1},z_{2})]^{2}$, and $F(z_{1},z_{2})=\boldsymbol{1}%
\left\{ .8\leq u(1;z_{1},z_{2})\leq 1\right\} $. Hence the goal is to obtain
robust bounds on the second moment and the probability that the solution to
the stochastic ODE falls within the interval $[.8,1]$ at time $t=1$,
respectively. Note that the indicator function is not smooth, and thus more
collocation points are required to obtain an accurate approximation. We
choose a Legendre polynomial basis for both $Z_{1}$ and $Z_{2}$ and use the
Legendre-Gauss quadrature points and weights for each dimension.

Figures \ref{ex1_1a} and \ref{ex1_2a} show the approximations as a function
of $c$ for $F(u)=u^{2}$ and $F(u)=\boldsymbol{1}\{.8\leq u\leq 1\},$
respectively. As expected $\Lambda _{c}^{1}$ gives the best bounds, while
the original form $\Lambda _{c}$ gives the worst upper bounds of the three.


\begin{figure}[h!]
\begin{center}
\includegraphics[width=4.42in,height=2.42in]{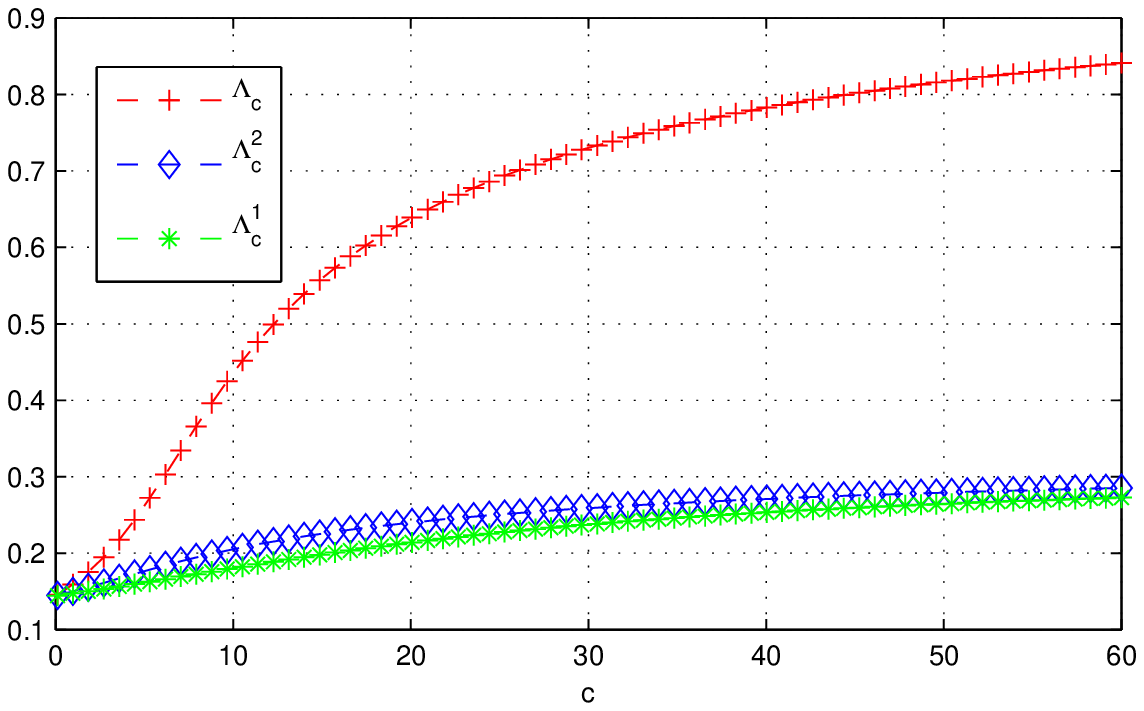}
\end{center}
\caption{ 
Example \protect\ref{example1}, $
\left\{ \Lambda _{c}^{i}\right\} _{i=1}^{2}$ for $F(u)=u^{2}$.
}
\label{ex1_1a}
\end{figure}


\begin{figure}[h!]
\begin{center}
\includegraphics[width=4.47in,height=2.45in]{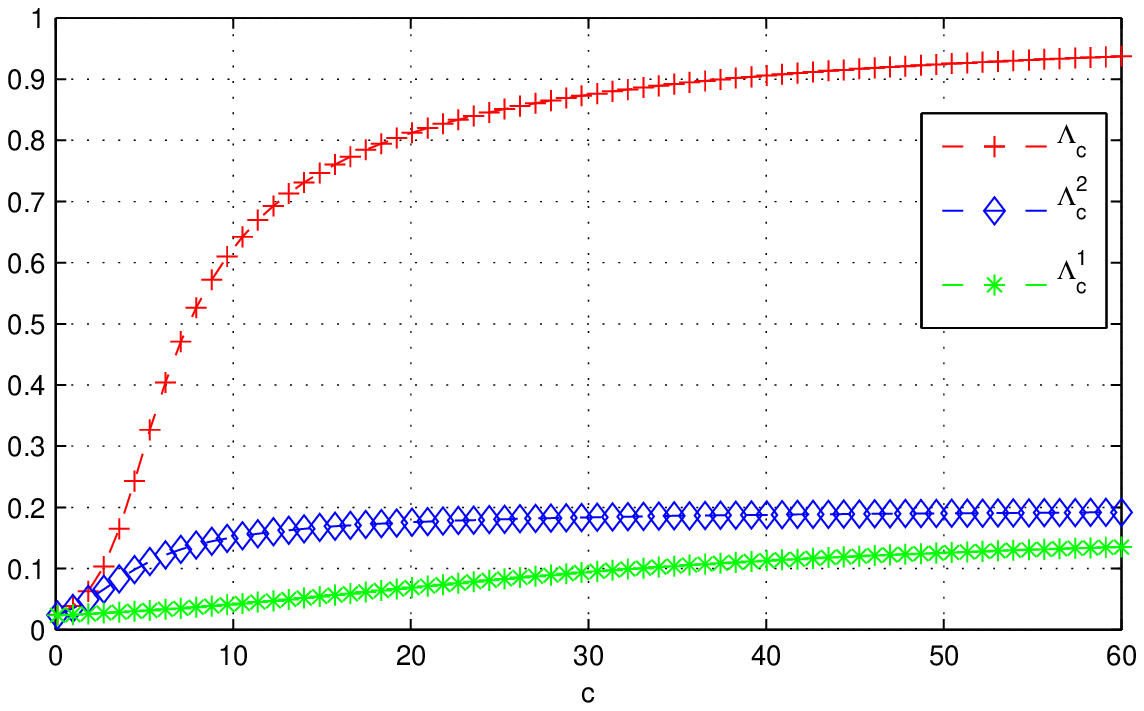}
\end{center}
\caption{
Example \protect\ref{example1}, 
$\Lambda _{c},\left\{ \Lambda _{c}^{i}\right\} _{i=1}^{2}$ for $F(u)=%
\boldsymbol{1}\{.8\leq u\leq 1\}$.
}
\label{ex1_2a}
\end{figure}

The plots also suggest what happens as $c\rightarrow \infty $. The measure $%
\Lambda _{c}$ is expected to yield robust bounds if one is uncertain
regarding the distributions of both random variables $Z_{1}$ and $Z_{2}$,
and $\lim_{c\rightarrow \infty }\Lambda _{c}$ represents the tightest upper
bound on performance if we know nothing about these random variables except
for their support. In this case, the limit in Figure \ref{ex1_2a} appears to
be 1, which means that for some choice of $z_{1}$ and $z_{2}$, $%
u(1;z_{1},z_{2})\in \lbrack .8,1]$. Hence without more information on the
distributions of $Z_{1},Z_{2}$ we do not obtain information other than this
from this performance bound. On the other hand, $\lim_{c\rightarrow \infty
}\Lambda _{c}^{i}$ for $i=1,2$ is strictly less than 1 and thus by Theorem %
\ref{thm:c_to_infty} gives a meaningful performance bound when the only
information available about $Z_{2}$ is its support. Note also that while $%
\Lambda _{c}^{1}$ and $\Lambda _{c}^{2}$ seem close for $F(u)=u^{2}$, the
differ considerably for $F(u)=\boldsymbol{1}\{.8\leq u\leq 1\}$.

In Figures \ref{ex1_1b} and \ref{ex1_2b} we display the relative error $%
|\Lambda _{c}^{i}-\Lambda _{c,exact}^{i}|/\left\vert \Lambda
_{c,exact}^{i}\right\vert $ between the three risk sensitive integrals and
the exact solution, which we were able to calculate using partially
analytical solutions. The relative error is fairly small for $F(u)=u^{2}$,
but significantly larger for $F(u)=\boldsymbol{1}\{.8\leq u\leq 1\}$, even
though more collocation points were used to determine the gPC coefficients.


\begin{figure}[h!]
\begin{center}
\includegraphics[width=4.32in,height=2.37in]{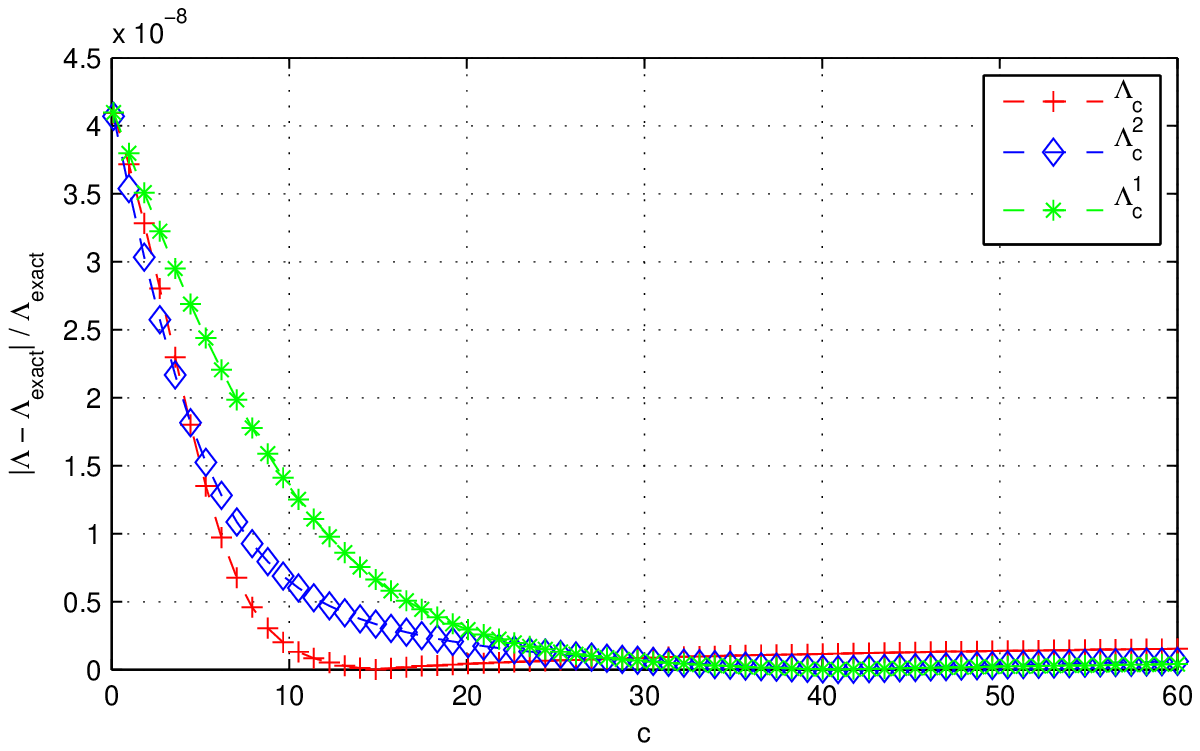}
\end{center}
\caption{
Example \protect\ref{example1},
relative error for $\Lambda _{c},\ \left\{ \Lambda _{c}^{i}\right\}
_{i=1}^{2}$ with $F(u)=u^{2}$.
}
\label{ex1_1b}
\end{figure}


\begin{figure}[h!]
\begin{center}
\includegraphics[width=4.35in,height=2.38in]{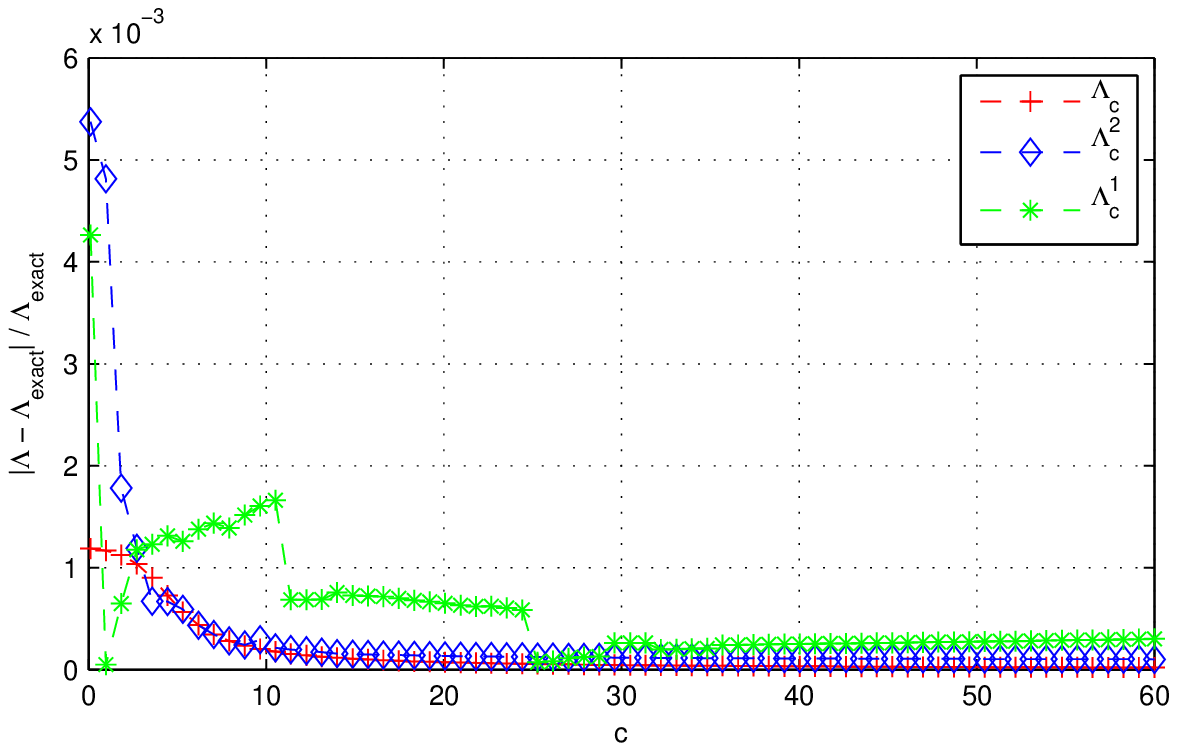}
\end{center}
\caption{
Example \protect\ref{example1},
relative error for $\Lambda _{c},\ \left\{ \Lambda _{c}^{i}\right\}
_{i=1}^{2}$ with $F(u)=\boldsymbol{1}\{.8\leq u\leq 1\}$.
}
\label{ex1_2b}
\end{figure}

Next we consider the problem of computing the value of $c$ which minimizes $%
c\rightarrow \frac{1}{c}B+\ \Lambda _{c}^{i}$ for $i=0,1,2$.\footnote{%
Note that the risk sensitive integrals are infinite for $c=0$. Thus in the
numerical calculations, we actually compute $c$ starting from $.01$.} Here $%
B $ is a maximum relative entropy distance that will be allowed between the
\textquotedblleft true\textquotedblright\ distribution, $\theta (dz_{2})$,
and the nominal distribution, $\gamma (dz_{2})$, and the minimization
produces the tightest possible bounds. For the example, suppose that the
family of distributions for which bounds are to be valid includes $\theta
(dz_{2})\sim $ beta$(\alpha =1.5,\beta =1.5)$, which implies $B\approx .0484$%
. In this example $B\leq f(c)\doteq c\frac{d}{dc}H^{i}(c)-H(c)$ for some $c$%
, where $H^{i}(c)\doteq c\Lambda _{c}^{i}$. As shown in Proposition \ref%
{prop2}, this implies there is a unique local minimum for $c\in (0,\infty )$%
. See Figures \ref{ex1_1c} and \ref{ex1_2c}. Alternatively, if $B>f(c)$ for
all $c$, then Proposition \ref{prop2} implies the minimum is achieved in the
limit as $c\rightarrow \infty $. From Figure \ref{ex1_2c} we find a minimum
for $\frac{1}{c}B+\ \Lambda _{c}^{1}$ of approximately 0.04 at $c\approx
5.12 $. Thus for all distributions on $Z_{2}$ whose relative entropy
distance to $U[0,1]$ is less than 0.0484, the probability that the solution
to the random ODE falls between .8 and 1 at time 1 is less than 0.04. 


\begin{figure}[h!]
\begin{center}
\includegraphics[width=4.47in,height=2.45in]{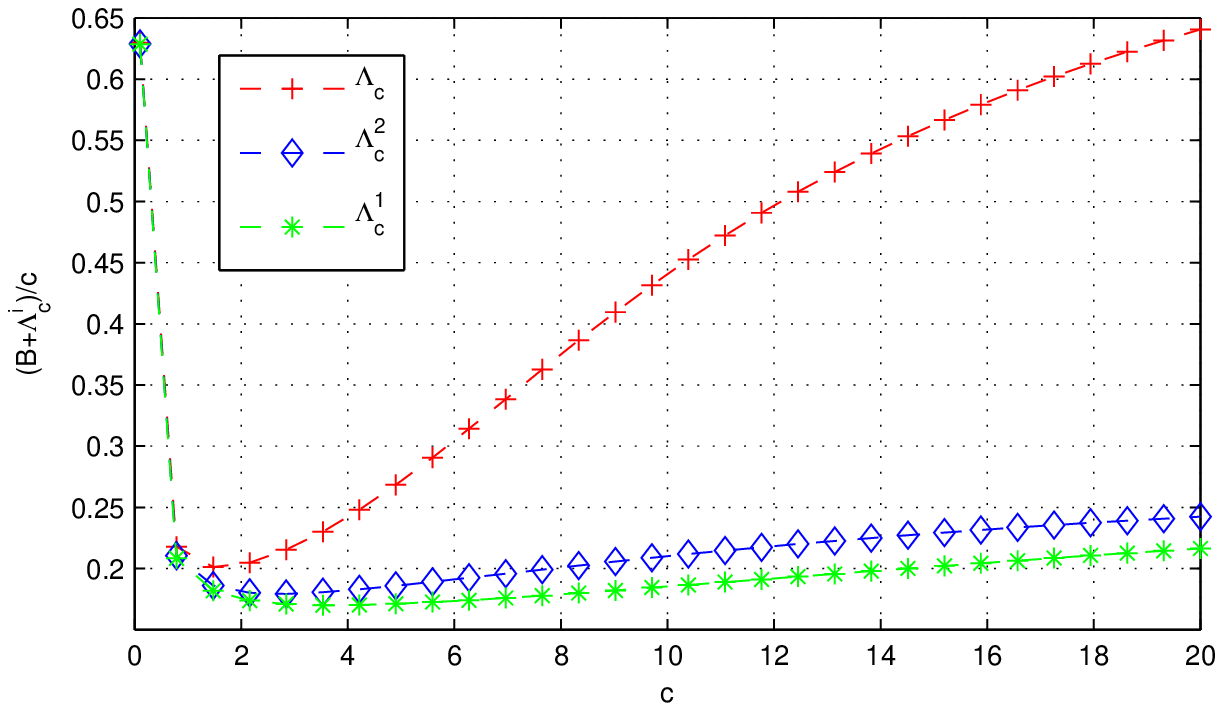}
\end{center}
\caption{
Example \protect\ref{example1}, $\frac{%
1}{c}(B\ +\ \Lambda _{c}),\left\{ \frac{1}{c}(B\ +\ \Lambda
_{c}^{i})\right\} _{i=1}^{2}$ with $F(u)=u^{2}$.
}
\label{ex1_1c}
\end{figure}


\begin{figure}[h!]
\begin{center}
\includegraphics[width=4.47in,height=2.45in]{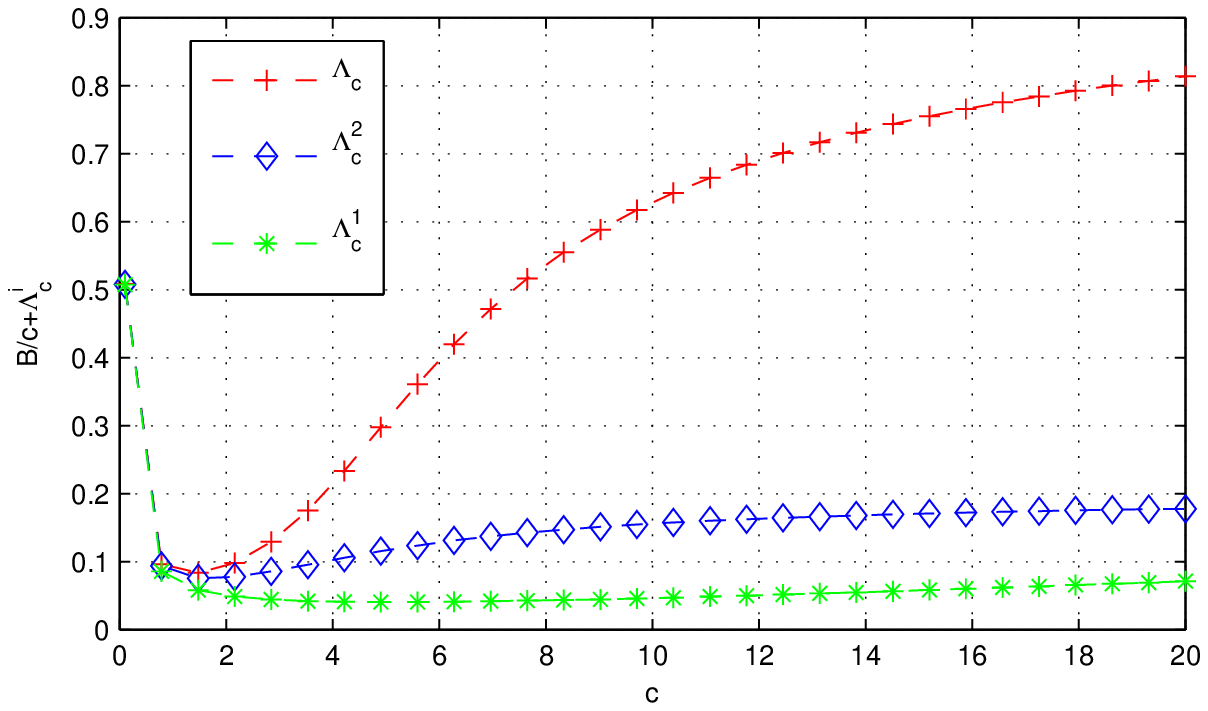}
\end{center}
\caption{
Example \protect\ref{example1}, $\frac{1}{c}(B\ +\
\Lambda _{c}),\left\{ \frac{1}{c}(B\ +\ \Lambda _{c}^{i})\right\} _{i=1}^{2}$
with $F(u)=\boldsymbol{1}\{.8\leq u\leq 1\}$.
}
\label{ex1_2c}
\end{figure}

Note that the minimum is at different values of $c$ for the three different
integrals, but the minimum is always the smallest for the first form of the
hybrid. In the case when the minimum occurs at some $c<\infty $, the minimum
is easily calculated, since this is a one-dimensional unconstrained
minimization problem (for example, one can use a golden method search
algorithm, such as MATLAB's fminbnd function). In the case when the minimum
occurs in the limit as $c\rightarrow \infty $, in order to find the minimum,
one can iterate $\Lambda _{c}^{i}$ until the successive iterations differ by
less than a prescribed tolerance.

The second example is a random oscillator, where the stochastic parameters
are the spring and dampening coefficients.

\begin{example}
\label{example2}\emph{We consider }$F(Z_{1},Z_{2})$\emph{\ defined in terms
of the solution of the ODE}%
\begin{equation*}
\frac{d^{2}}{dt^{2}}u(t)+\gamma (z_{2})\frac{d}{dt}u(t)+k(z_{1})u(t)=f(t),%
\quad u(0)=u_{0},\quad u^{\prime }(0)=u_{0}^{\prime },
\end{equation*}%
\emph{where }$k$\emph{\ and }$\gamma $\emph{\ are known functions, which
represent the spring and dampening coefficients, respectively. The outcome
of interest is whether or not the position of the random oscillator falls
within a specified range, }$[a,b]$\emph{, at a specified time, }$%
t_{critical} $\emph{. Letting }$u(t;z_{1},z_{2})$\emph{\ denote the solution
parameterized by }$(z_{1},z_{2})$\emph{, set }$F(Z_{1},Z_{2})=\mathbf{1}%
\{a\leq u(t_{critical};Z_{1},Z_{2})\leq b\}$\emph{.}
\end{example}

Here we consider the random oscillator with $k(z_{1})=4z_{1},\gamma
(z_{2})=z_{2}/5,\ $and $f(t)=10\cos (10t)+3$, and choose $Z_{1}\sim $ beta$%
(\alpha =5,\beta =5),\ Z_{2}\sim $ beta$(\alpha =5,\beta =5)$, $%
t_{critical}=4$, and $[a,b]=(-\infty ,2]$. Here we use the Gauss-Jacobi
quadrature points and weights. In this example we are concerned with the
position of the oscillator at a critical time. Figure \ref{ex2_1a} plots the
risk sensitive integrals for Example \ref{example2} with $F(u)=\boldsymbol{1}%
\{u\leq 2\}$. Figure \ref{ex2_1c} depicts $c\rightarrow \frac{1}{c}%
(B+\Lambda _{c}^{i})$, where $B=R\left( \theta (dz_{2})\left\Vert \gamma
(dz_{2})\right. \right) $ $\approx .0587$, with $\theta (dz_{2})\sim $ beta$%
(\alpha =10,\beta =10)$.


\begin{figure}[h!]
\begin{center}
\includegraphics[width=4.45in,height=2.43in]{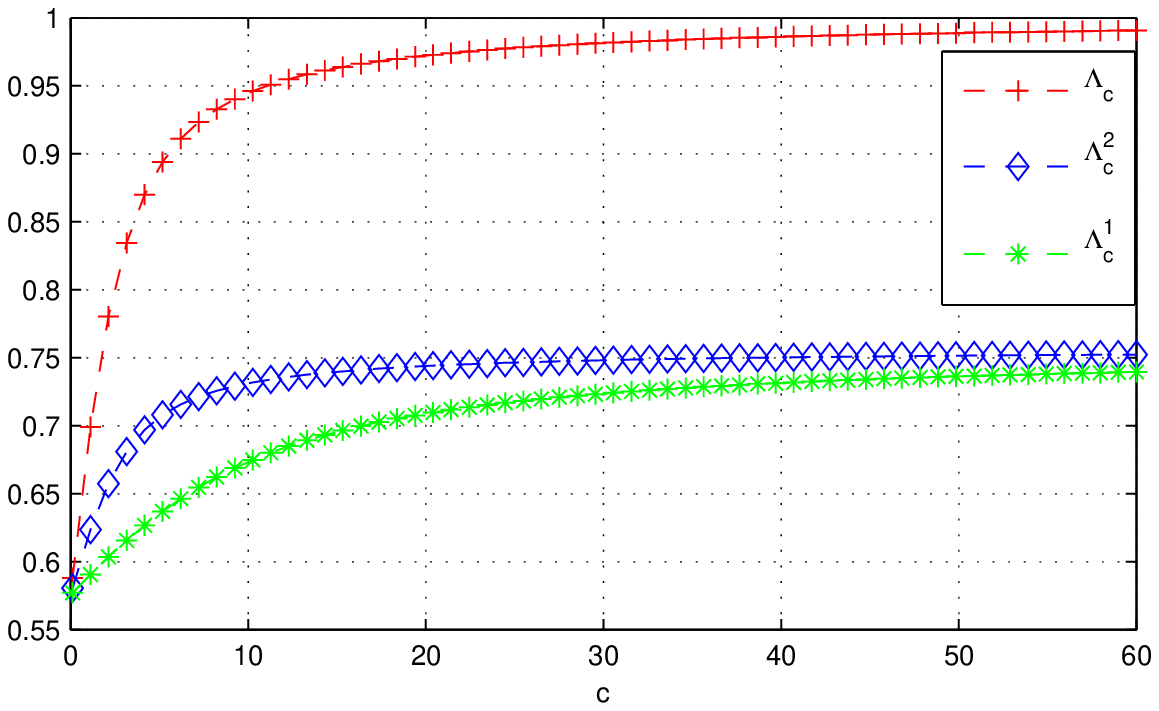}
\end{center}
\caption{
Example \protect\ref{example2}, $%
\Lambda _{c},\left\{ \Lambda _{c}^{i}\right\} _{i=1}^{2}$ for $F(u)=%
\boldsymbol{1}\{u\leq 2\}$.
}
\label{ex2_1a}
\end{figure}


\begin{figure}[h!]
\begin{center}
\includegraphics[width=4.47in,height=2.45in]{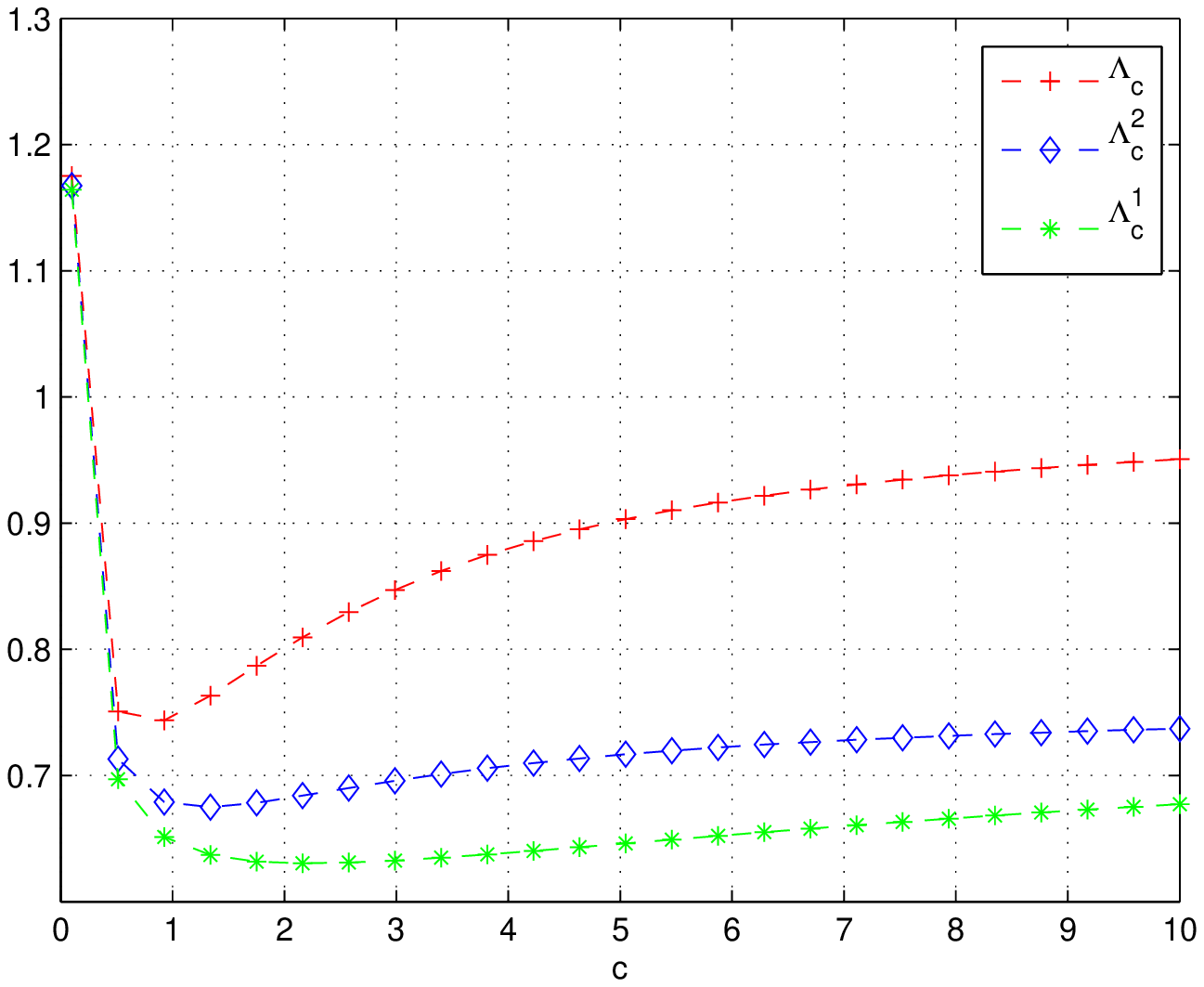}
\end{center}
\caption{
Example \protect\ref{example2}, $%
\frac{1}{c}(B\ +\ \Lambda _{c}),\left\{ \frac{1}{c}(B\ +\ \Lambda
_{c}^{i})\right\} _{i=1}^{2}$ with $F(u)=\boldsymbol{1}\{u\leq 2\}$.
}
\label{ex2_1c}
\end{figure}

Finally, the third example involves a heat equation with random heat
capacity and thermal conductivity.

\begin{example}
\label{example3}\emph{We consider }$F(Z_{1},Z_{2})$\emph{\ defined in terms
of the solution of the PDE}%
\begin{equation*}
\frac{d}{dt}u(t,x)=\frac{k(u;z_{1})}{m(z_{2})}\frac{d^{2}}{dx^{2}}u(t,x),\ \
\ \ \ u(0,x)=u_{0}(x)
\end{equation*}%
\emph{on }$x\in (0,L)$\emph{\ with boundary conditions}%
\begin{equation*}
-k(u;z_{1})\frac{d}{dx}u(t,0)=q,\ \ \ \frac{d}{dx}u(t,L)=0.\ 
\end{equation*}%
\emph{Here }$k$\emph{\ and }$m$\emph{\ are the thermal conductivity and heat
capacity, respectively. Note that in this example, the randomness affects
the diffusivity and the boundary conditions. We are interested in the
probability that the material exceeds a particular temperature, }$%
T_{critical}$\emph{, at the time }$t_{final}$\emph{\ and at some point }$%
x^{\star }\in \lbrack 0,L]$\emph{. Letting }$u(t;z_{1},z_{2})$\emph{\ denote
the solution parameterized by }$(z_{1},z_{2})$\emph{, set }$F(Z_{1},Z_{2})=%
\boldsymbol{1}\{u(t_{final},x^{\star };Z_{1},Z_{2})\geq T_{critical}\}$\emph{%
.}
\end{example}

Here we consider the random heat equation with (nonlinear) Neumann boundary
conditions. We use $k(u,z_{1})=z_{1}+(1.5\times 10^{-7})u,$ $%
m(z_{2})=(10^{-6})z_{2}$, $q=.35$, $L=1.90$, $u_{0}(x)=25$, and $%
T_{critical}=980$, $x^{\star }=0$, $t_{final}=1000$, and set $F(Z_{1},Z_{2})=%
\mathbf{1}\{u(t_{final},x^{\star };Z_{1},Z_{2})\geq T_{critical}\}$. For the
random variables, we introduce two independent random variables $\tilde{Z}%
_{1}\sim $ beta$(\alpha =5,\beta =5),\ \tilde{Z}_{2}\sim $ beta$(\alpha
=5,\beta =5)$ and let $Z_{1}=(2\times 10^{-3})\tilde{Z}_{1}\ +\ 3\times
10^{-3},$ $Z_{2}=(0.11)\tilde{Z}_{2}+0.30$. Figures \ref{ex3_a} and \ref%
{ex3_b} plot the risk sensitive integrals as a function of $c$ and
illustrate $c\rightarrow \frac{1}{c}(B+\Lambda _{c}^{i})$, where $B=R\left(
\theta (dz_{2})\left\Vert \gamma (dz_{2})\right. \right) \approx .0587$. and
with $\theta (dz_{2})\sim $ beta$(\alpha =10,\beta =10)$. Again, we use the
Gauss-Jacobi quadrature points and weights.


\begin{figure}[h!]
\begin{center}
\includegraphics[width=4.47in,height=2.45in]{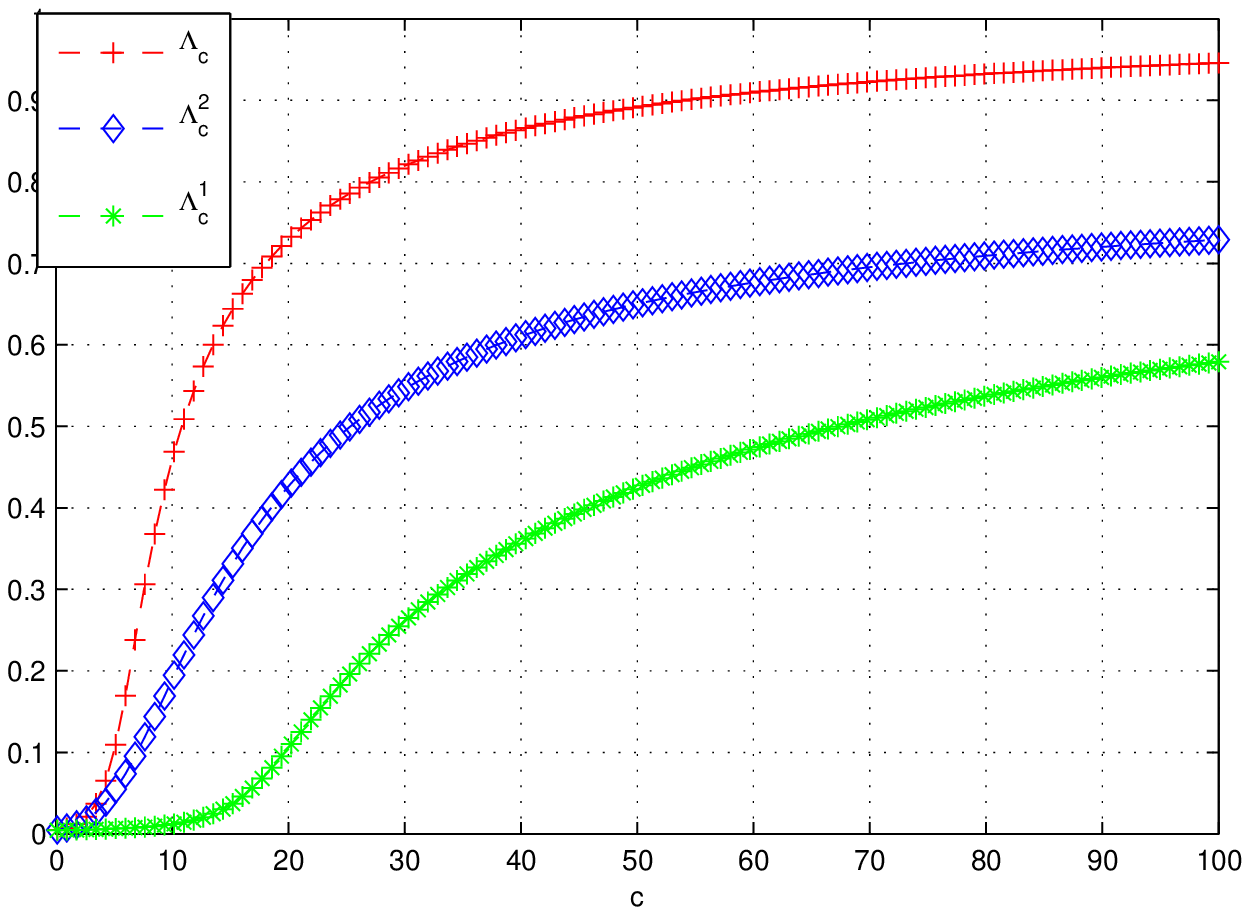}
\end{center}
\caption{
Example \protect\ref{example3},
risk sensitive integrals as a function of $c$. 
}
\label{ex3_a}
\end{figure}


\begin{figure}[h!]
\begin{center}
\includegraphics[width=4.47in,height=2.45in]{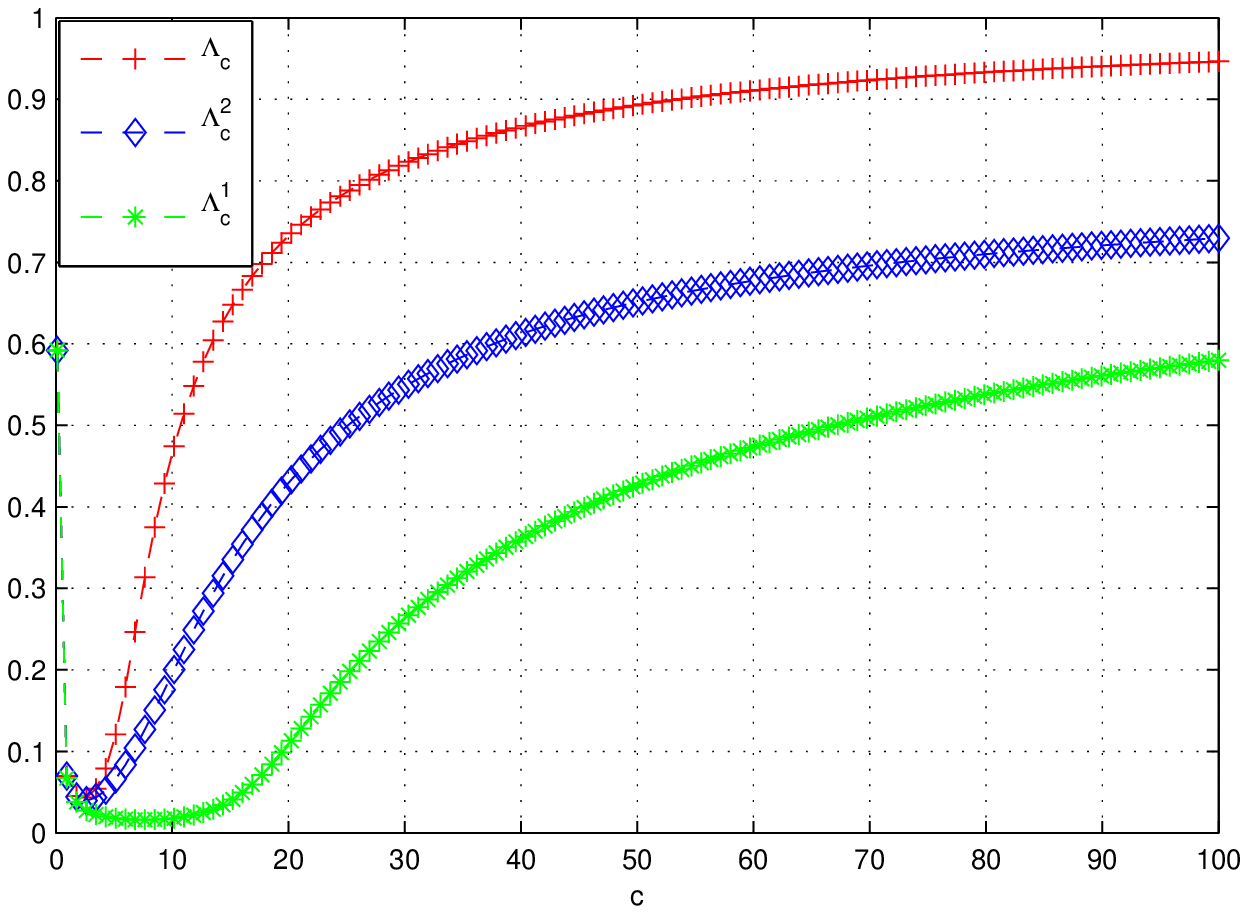}
\end{center}
\caption{
Example \protect\ref{example3},
plot of $\frac{1}{c}(B+\Lambda _{c}),\ \left\{ \frac{1}{c}(B+\Lambda
_{c}^{i})\right\} _{i=1}^{2}$. 
}
\label{ex3_b}
\end{figure}

\subsubsection{Dependent aleatoric and epistemic uncertainties}

\label{subsub:depend_uncertainties}

In this section we consider problems where the distribution of $Z_{1}$ can
depend on the value of $Z_{2}$, or vice versa. In the case when $Z_{1}$
depends on $Z_{2}$, we use the following inequality, which holds for the
alternative of the first hybrid form (see the discussion in Section \ref%
{Section:RSforms}):%
\begin{equation*}
\int_{\mathcal{Z}_{2}}\int_{\mathcal{Z}_{1}}F(z_{1},z_{2})\mu
(dz_{1}|z_{2})\theta (dz_{2})\leq \frac{1}{c}R\left( \theta
(dz_{2})\left\Vert \gamma (dz_{2})\right. \right) +\bar{\Lambda}_{c}^{1},
\end{equation*}%
where%
\begin{equation*}
\bar{\Lambda}_{c}^{1}=\frac{1}{c}\log \int_{\mathcal{Z}_{2}}\exp \left(
\int_{\mathcal{Z}_{1}}cF(z_{1},z_{2})\mu (dz_{1}|z_{2})\right) \gamma
(dz_{2})\text{.}
\end{equation*}%
Recall that the relative entropy term $\frac{1}{c}R\left( \theta
(dz_{2})\left\Vert \gamma (dz_{2})\right. \right) $ represents our lack of
knowledge about the distribution of $Z_{2}$, i.e., the epistemic uncertainty
where we want robustness, whereas $Z_{1}~$represents aleatoric uncertainty.
A question that now arises is whether or not it is still possible to use gPC
methods to evaluate this integral. Consider the example with $Z_{2}\sim
U[0,1]$ and $Z_{1}\sim N(Z_{2},1)$, so that $Z_{1}$ is Gaussian with
variance 1 and mean $Z_{2}$. Certainly $Z_{1}$ and $Z_{2}$ are not
independent random variables. However, we can introduce an auxiliary normal
random variable $Z\sim \mathcal{N}(0,1)$, and write $Z_{1}\doteq Z+Z_{2}$,
where $Z_{2}$ and $Z$ are independent random variables. Now we consider $%
G(Z,Z_{2})$ instead of $F(Z_{1},Z_{2})$, and note that $%
G(Z,Z_{2})=F(Z+Z_{2},Z_{2})$. Furthermore, $\bar{\Lambda}_{c}^{1}$ can be
written as%
\begin{equation*}
\bar{\Lambda}_{c}^{1}=\frac{1}{c}\log \int_{\mathcal{Z}_{2}}\exp \left(
\int_{\mathcal{Z}}cG(z,z_{2})\mu (dz)\right) \gamma (dz_{2}).
\end{equation*}%
Hence, $\bar{\Lambda}_{c}^{1}$ can be calculated just like $\Lambda _{c}^{1}$
via a gPC approximation. Similar results hold for the case when $Z_{2}$
depends on $Z_{1}$, but the epistemic uncertainty is in $Z_{2}$.

In general, as long as we can find a smooth, invertible mapping from $%
(Z_{1},Z_{2})$ to an independent pair of random variables, one can evaluate
the integrals as before. Figures \ref{ex1_3a} and \ref{ex1_3c} show the
performance of risk sensitive integrals for dependent random variables as
above for Example \ref{example1} and with $F(u)=\boldsymbol{1}\{1/2\leq
u\leq 1\}$.


\begin{figure}[h!]
\begin{center}
\includegraphics[width=4.47in,height=2.45in]{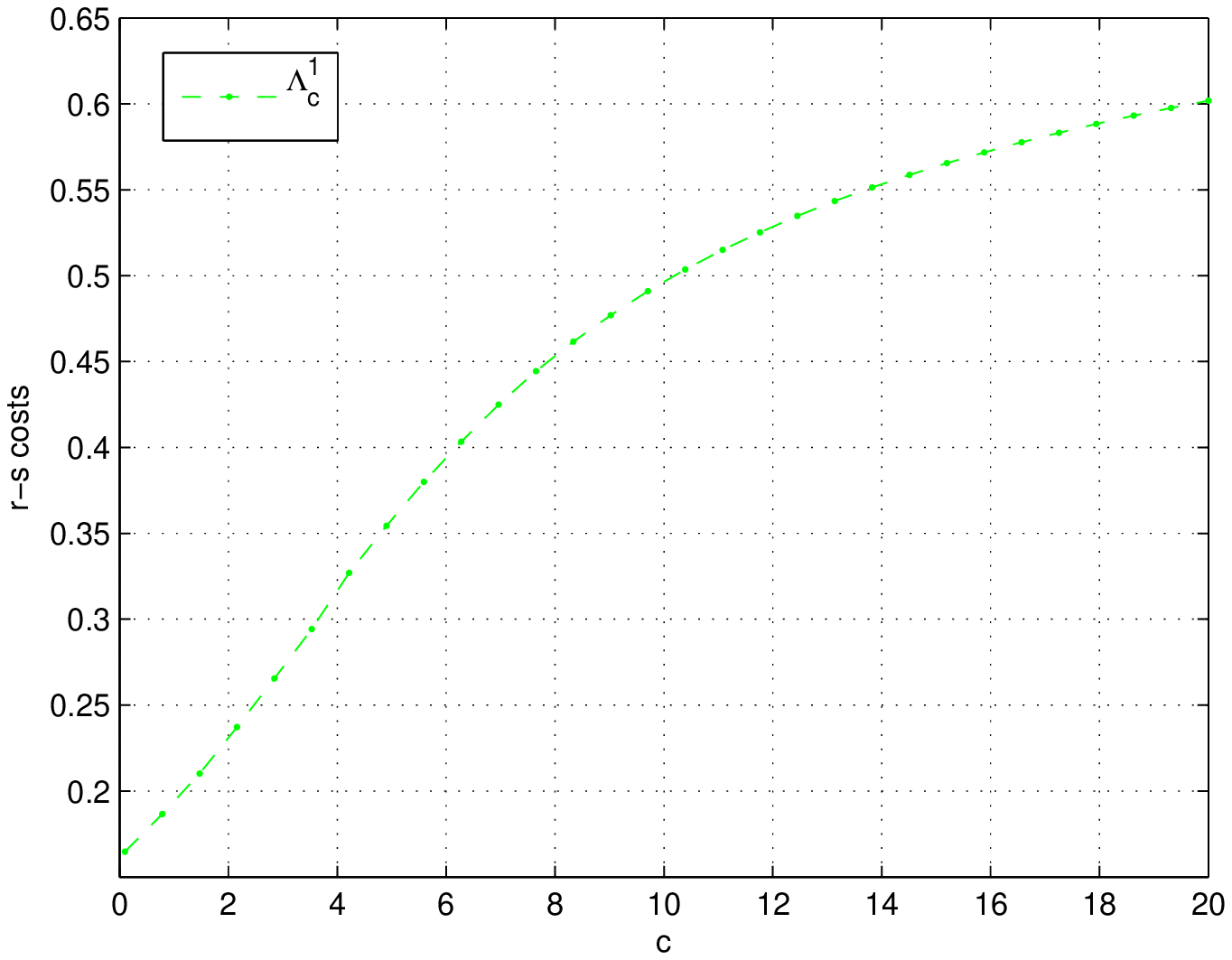}
\end{center}
\caption{
Example \protect\ref{example1}, 
$\bar{\Lambda}_{c}^{1}$ for $F(u)=\boldsymbol{1}\{1/2\leq u\leq 1\}.$
}
\label{ex1_3a}
\end{figure}


\begin{figure}[h!]
\begin{center}
\includegraphics[width=4.47in,height=2.45in]{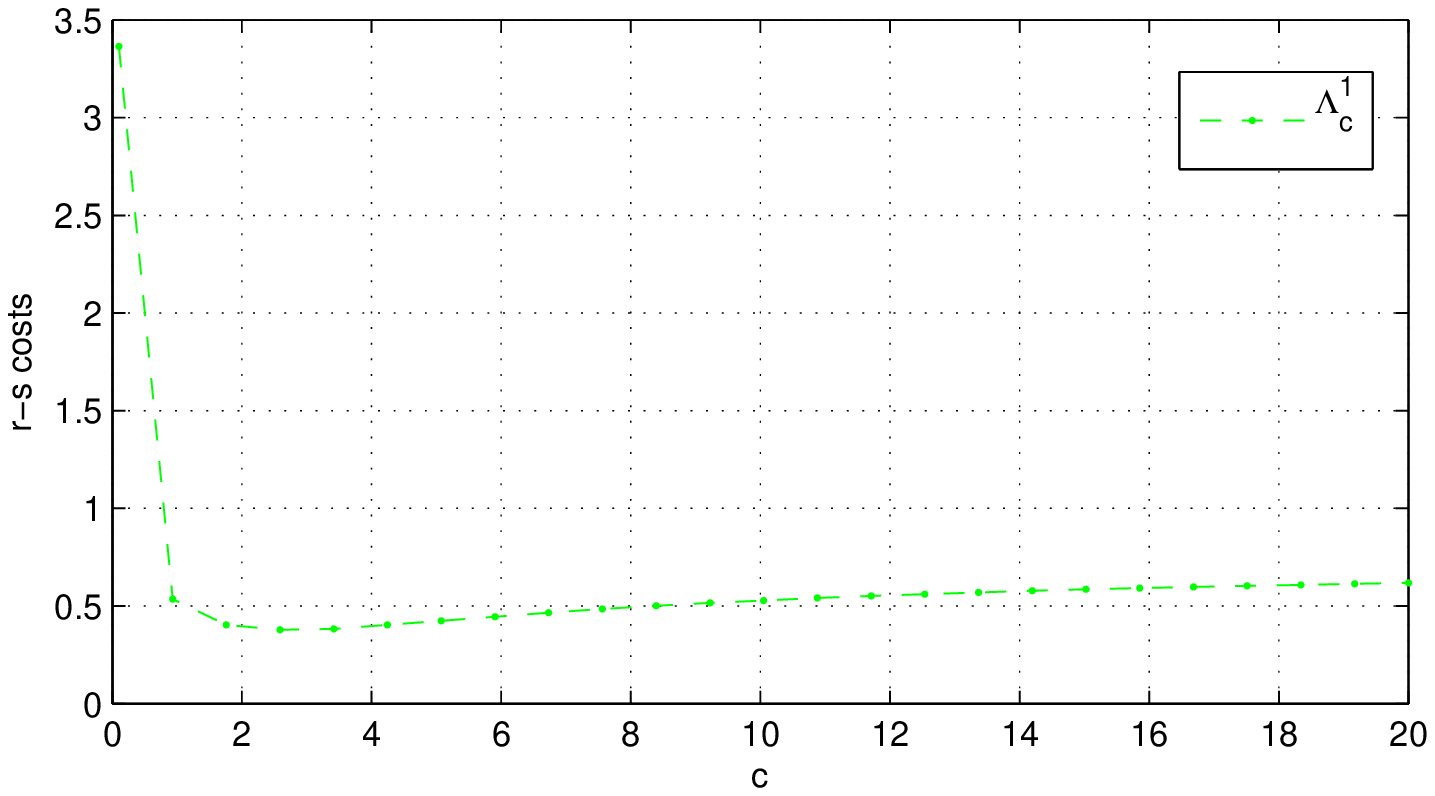}
\end{center}
\caption{
Example \protect\ref{example1},$%
\frac{1}{c}(B+\bar{\Lambda}_{c}^{1})$ $F(u)=\boldsymbol{1}\{1/2\leq u\leq
1\} $, and $\protect\theta \sim \mathcal{N(}.8,1)$.
}
\label{ex1_3c}
\end{figure}

\subsection{Discussion and Conclusion}

\bigskip

In a recent paper by Xiu et. al. \cite{xiueld}, an approach to dealing with
epistemic uncertainty via polynomial chaos methods was developed. In \cite%
{xiueld}, epistemic uncertainty meant that the true distribution of the
underlying stochastic parameters was not known precisely. Here we compare
our method of handling this type of epistemic uncertainty to theirs.

The approach taken in \cite{xiueld} is based on stochastic collocation
approximation. Specifically, the differential equation of interest was
solved on a generic collocation grid on the sample space of the random
variables. This is in contrast to standard uses of stochastic collocation
for uncertainty quantification, where the collocation grid is chosen in
accordance with the (assumed known) underlying distribution.\ Of course the
reason for the particular choice of basis polynomials and quadrature points
is to obtain optimal (spectral) convergence with respect to the $L^{2}$
weighted error. In the approach of \cite{xiueld}, regardless of what is the
\textquotedblleft true\textquotedblright\ (and assumed unknown)
distribution, one chooses grid points with respect to the $L^{2}$ weighted
norm with constant weight (e.g., Legendre-Gauss quadrature weights), even if
the original space is unbounded. One then solves the differential equation
at these collocation points and constructs the interpolating polynomial of
the approximation. Hence, the gPC approximation converges optimally in the $%
L^{2}$ norm only if the underlying distribution is uniform. For non-uniform
random variables the approximation will still converge point-wise as the
number of collocation points increases (see Section 5 in \cite{xiueld}).

Once the interpolated approximation is obtained, one can determine the mean,
variance, probabilities, and so on with respect to various underlying
distributions on the parameters via Monte Carlo or other methods. The
benefits are still large, in that one need not evaluate the differential
equation for every Monte Carlo sample, but instead evaluates a much simpler
approximation via a finite sum of polynomial basis functions. However,
information can be obtained only by fixing a choice for the underlying
distribution. Also, depending on which underlying distribution is used, the
errors in the function approximation will have a greater or smaller effect.

In this paper we have developed an approach in which performance bounds can
be calculated for the mean, variance, probabilities, and so on with respect
to all distributions that are within a particular relative entropy distance
from the nominal distribution, and the bounds are optimal for that class of
distributions. In addition to providing performance bounds on a class of
distributions defined by their relative entropy distance, the limit $%
c\rightarrow \infty $ is particularly interesting in that it gives tight
bounds on epistemic uncertainties where one assumes nothing about the
underlying distribution of certain variables except their support. In
addition, the method one can efficiently handle both aleatoric and epistemic
uncertainty simultaneously. More precisely, by computing certain hybrid
forms of a risk-sensitive expectation we obtain tight performance bounds
that distinguish between variables with known distribution and variables
with unknown distribution or other forms of epistemic uncertainty. In
particular, the scenarios to which the method applies include (i) aleatoric
with known distribution; (ii) aleatoric with partly known distribution
(mingled aleatoric and epistemic); (iii) epistemic for which one is willing
to model by a family of aleatoric uncertainties, and (iv) epistemic where
one is only willing to place bounds on the uncertainties. For all the
examples we have considered (including those presented in the paper),
approximation of the required risk-sensitive integrals can be done via gPC
methods using roughly the same computational effort as would be required to
compute the corresponding ordinary performance measures using the nominal
probability distributions.

\appendix

\section{\protect\bigskip Distribution for Polynomial Basis}

Here we list some of the most common distributions, both continuous and
discrete, and their corresponding gPC basis representations (see \cite{xiu1}%
).

\begin{equation}
\begin{tabular}{cc}
\hline
\strut \rule{0pt}{12pt}\textbf{Distribution} & \textbf{gPC polynomial basis}
\\ 
Gaussian & Hermite \\ 
Gamma & Laguerre \\ 
Beta & Jacobi \\ 
Uniform & Legendre \\ 
Poisson & Charlier \\ 
Binomial & Krawtchouk \\ 
Negative Binomial & Meixner \\ 
Hypergeometric & Hahn \\ \hline
\end{tabular}
\label{gpc_table}
\end{equation}%
\bigskip

\section{\protect\bigskip Relative Entropy Formulas}

In this section we evaluate relative entropies for some of the most common
types of distributions listed in Table \ref{gpc_table}. Since most of the
distributions listed in this table fall within the exponential family, we
start with the general form%
\begin{equation}
f_{X}(x;\boldsymbol{\theta })=h(x)\exp \left( \sum_{i=1}^{s}\eta _{i}(%
\boldsymbol{\theta }\mathbf{)}T_{i}(x)-A(\eta _{1}(\boldsymbol{\theta }%
\mathbf{)},...,\eta _{s}(\boldsymbol{\theta }\mathbf{)})\right) ,
\label{expfamily}
\end{equation}%
where $T_{i}(x),h(x),\eta _{i}(\boldsymbol{\theta })$, and $A(\boldsymbol{%
\theta })$ are known functions, $\boldsymbol{\theta }=(\theta
_{1},...,\theta _{d})^{T}$ is the parameter of the family.

\begin{example}
\emph{For Gaussian distributions }$\theta =(\mu ,\sigma )^{T}$\emph{, where }%
$\mu $\emph{\ is the mean and }$\sigma $\emph{\ is the variance. This is put
in the general form by setting}%
\begin{equation*}
\boldsymbol{\theta }=\left( \mu ,\sigma \right) ^{T},\quad \eta =\left( 
\frac{\mu }{\sigma ^{2}},\frac{-1}{2\sigma ^{2}}\right) ^{T},\quad h(x)=%
\frac{1}{\sqrt{2\pi }},\quad T(x)=(x,x^{2})^{T},
\end{equation*}%
\begin{equation*}
A(\eta (\boldsymbol{\theta }))=\frac{\mu ^{2}}{2\sigma ^{2}}+\ln |\sigma |=-%
\frac{\eta _{1}^{2}}{4\eta _{2}}+\frac{1}{2}\ln \left\vert \frac{1}{2\eta
_{2}}\right\vert ,
\end{equation*}%
\emph{producing} 
\begin{equation*}
f_{X}(x;\mu ,\sigma )=\frac{1}{\sqrt{2\pi \sigma ^{2}}}e^{-(x-\mu
)^{2}/2\sigma ^{2}}.
\end{equation*}
\end{example}

We now state the relative entropy formula for distributions which have a
probability density function given by the form of exponential family in (\ref%
{expfamily}). The relative entropy between a probability measure $P$ with
density $p(x)$ and another probability measure $Q$ with density $q(x)$ is
given by%
\begin{equation}
R(P\left\Vert Q\right. )=\int_{\Gamma }p(x)\ln \frac{p(x)}{q(x)}dx
\label{relative_entropy}
\end{equation}%
whenever $P$ is absolutely continuous with respect to $Q$, where $\Gamma $
is the support of $p$ and $q$. In all other cases it is defined to be $%
\infty $.

A simplified version of the relative entropy formula for distributions from
the same exponential family type is available. Suppose that $q(x)$ and $p(x)$
have the same $h,\eta _{i},T_{i},A$, but with different parameters $%
\boldsymbol{\theta }_{1}$ and $\boldsymbol{\theta }_{2}$ (note that $%
\boldsymbol{\theta }_{1}$ and $\boldsymbol{\theta }_{2}$ are vectors). If we
associate $q(x)$ with $\boldsymbol{\theta }_{1}$ and $p(x)$ with $%
\boldsymbol{\theta }_{2}$, then it can be shown that the relative entropy
formula reduces to%
\begin{equation}
R(P\left\Vert Q\right. )=\sum_{i=1}^{s}(\eta _{i}(\boldsymbol{\theta }%
_{2})-\eta _{i}(\boldsymbol{\theta }_{1}))\int_{\Gamma }T_{i}(x)p(x)dx-A(%
\boldsymbol{\eta }\mathbf{(}\boldsymbol{\theta }\mathbf{_{2}))}+A(%
\boldsymbol{\eta }\mathbf{(}\boldsymbol{\theta }_{1})).  \label{RE_expfam}
\end{equation}%
\medskip

\noindent \textbf{Gaussian.} If $Q=N(\mu _{1},\sigma _{1})$ and $P=N(\mu
_{2},\sigma _{2})$, then using the fact that $\int T_{1}(x)p(x)dx=\mu _{2}$
and $\int T_{2}(x)p(x)dx=\mu _{2}^{2}+\sigma _{2}^{2}$, we get%
\begin{equation}
R(P\left\Vert Q\right. )=\frac{1}{2\sigma _{1}^{2}}\left( (\mu _{1}^{2}-\mu
_{2}^{2})\ +\ (\sigma _{2}-\sigma _{1})\right) \ +\ln \frac{\sigma _{1}}{%
\sigma _{2}}\ .  \label{re_normal}
\end{equation}%
\medskip

\noindent \textbf{Beta.} For the beta distribution, the parameters in the
exponential family are given by%
\begin{equation*}
\eta =(\alpha -1,\beta -1)^{T},\ \ T=(\ln (x),\ln (1-x))^{T},\ \ h(x)=1,\ \
A=\ln \frac{\Gamma (\alpha )\Gamma (\beta )}{\Gamma (\alpha +\beta )}
\end{equation*}%
where 
\begin{equation*}
f_{X}(x;\alpha ,\beta )=\frac{\Gamma (\alpha )\Gamma (\beta )}{\Gamma
(\alpha +\beta )}x^{\alpha -1}(1-x)^{\beta -1},\qquad 0<x<1,\ \alpha ,\beta
>1.
\end{equation*}%
Another form of the beta distribution, which is more closely related to the
Jacobi\ polynomials, is given by%
\begin{equation*}
\tilde{f}_{X}(k;\tilde{\alpha},\tilde{\beta})=\frac{(1-x)^{\tilde{\alpha}%
}(1+x)^{\tilde{\beta}}}{2^{\tilde{\alpha}+\tilde{\beta}+1}B(\tilde{\alpha}+1,%
\tilde{\beta}+1)},\qquad -1<x<1,\ \tilde{\alpha},\tilde{\beta}>-1,
\end{equation*}%
where $B(\alpha ,\beta )=\Gamma (\alpha )\Gamma (\beta )/\Gamma (\alpha
+\beta )$ is the beta function. Note that the relation between the two is
simply a rescaling and a variable substitution given by $\tilde{\alpha}%
=\beta -1,\tilde{\beta}=\alpha -1$. We will use the primary form to
determine the relative entropy formula. Letting $P=$ beta$(\alpha _{2},\beta
_{2})$ and $Q=$ beta$(\alpha _{1},\beta _{1})$,%
\begin{equation*}
\int_{0}^{1}T_{1}(x)p(x)dx=\psi (\alpha _{2})-\psi (\alpha _{2}+\beta
_{2}),\ \ \int_{\Gamma }T_{2}(x)p(x)dx=\psi (\beta _{2})-\psi (\alpha
_{2}+\beta _{2}),
\end{equation*}%
where $\psi (x)$ is known as the digamma function, or the zero$^{th}$ order
of the polygamma function. Then,%
\begin{eqnarray*}
R(P\left\Vert Q\right. ) &=&(\alpha _{2}-\alpha _{1})(\psi (\alpha
_{2})-\psi (\alpha _{2}+\beta _{2})) \\
&&+\ (\beta _{2}-\beta _{1})(\psi (\beta _{2})-\psi (\alpha _{2}+\beta
_{2}))\ +\ \ln \frac{B(\alpha _{1},\beta _{1})}{B(a_{2},\beta _{2})}.
\end{eqnarray*}%
\medskip

\noindent \textbf{Gamma.} Here we have%
\begin{equation*}
\eta =(-\beta ,\alpha -1)^{T},\ \ T=(x,\ln (x))^{T},\ \ h(x)=1,\ \ A=\log
\Gamma (\alpha )-\alpha \log \beta
\end{equation*}%
If $P=$ gamma$(\alpha _{2},\beta _{2}),Q=$ gamma$(\alpha _{1},\beta _{1})$,
where $\alpha _{i},\beta _{i}$ are the so-called shape parameters, then 
\begin{equation*}
R(P\left\Vert Q\right. )=\frac{\alpha _{1}}{\beta _{1}}(\beta _{2}-\beta
_{1})+\ (\alpha _{1}-\alpha _{2})(\psi (\alpha _{1}))-\log (\beta
_{1}))+\log \frac{\Gamma (\alpha _{2})\beta _{1}^{\alpha _{1}}}{\Gamma
(\alpha _{1})\beta _{2}^{a_{2}}}\text{.}\ 
\end{equation*}%
\medskip \qquad

\noindent \textbf{Binomial.} Here 
\begin{equation*}
\eta =(-\beta ,\alpha -1)^{T},\ \ T=(x,\ln (x))^{T},\ \ h(x)=1,\ \ A=\log
\Gamma (\alpha )-\alpha \log \beta .
\end{equation*}%
If $P=$ binomial$(n,p_{2}),Q=$ binomial$(n,p_{1})$, then 
\begin{equation*}
R(P\left\Vert Q\right. )=\log \frac{p_{1}^{\mu _{1}}(1-p_{2})^{\mu _{1}-n}}{%
p_{2}^{\mu _{1}}(1-p_{1})^{\mu _{1}-n}},\qquad \mu _{1}=np_{1}.
\end{equation*}%
\medskip

\noindent \textbf{Poisson.} In this case 
\begin{equation*}
\eta =\log \lambda ,\ \ T=x,\ \ h(x)=\frac{1}{x!},\ \ A=\lambda
\end{equation*}%
If $P=$ Poisson$(\lambda _{2}),Q=$ Poisson$(\lambda _{1})$, then the
relative entropy distance is 
\begin{equation*}
R(P\left\Vert Q\right. )=\lambda _{1}-\lambda _{2}\ +\ \lambda _{2}\log 
\frac{\lambda _{2}}{\lambda _{1}}.
\end{equation*}%
\medskip

Note that relative entropy formula is invariant under shifting and scaling
of both distributions simultaneously. In fact, suppose that $\psi $ and its
inverse are both well defined and measurable, $X$ and $Y$ have distributions 
$P$ and $Q$, and that $\psi (X)$ and $\psi (Y)$ have distributions $\bar{P}$
and $\bar{Q}$. Then \cite[Lemma E.2.1]{dupell4} $R(\bar{P}\left\Vert \bar{Q}%
\right. )=R(P\left\Vert Q\right. )$. A list of relative entropy formulas
follows.%
\begin{equation*}
\begin{tabular}{l|l}
\textbf{Distribution} & $\mathbf{R((\cdot )}_{2}\left\Vert \mathbf{(\cdot )}%
_{1}\right. \mathbf{)}$ \\[0.1cm] \hline
Gaussian & $\frac{1}{2\sigma _{1}^{2}}\left[ (\mu _{1}^{2}-\mu
_{2}^{2})+(\sigma _{2}-\sigma _{1})\right] +\log \frac{\sigma _{1}}{\sigma
_{2}}\rule{0pt}{14pt}$ \\[0.15cm] 
Beta/Uniform & 
\begin{tabular}{l}
$(\alpha _{2}-\alpha _{1})\left[ \psi (\alpha _{2})-\psi (\alpha _{2}+\beta
_{2})\right] $ \\ 
$+(\beta _{2}-\beta _{1})\left[ \psi (\beta _{2})-\psi (\alpha _{2}+\beta
_{2})\right] +\log \frac{B(\alpha _{1},\beta _{1})}{B(\alpha _{2},\beta _{2})%
}$%
\end{tabular}
\\[0.4cm] 
Gamma & 
\begin{tabular}{l}
$\frac{\alpha _{1}}{\beta _{1}}(\beta _{2}-\beta _{1})+(\alpha _{1}-\alpha
_{2})\left[ \psi (\alpha _{1})-\log \beta _{1}\right] $ \\ 
$+\log \left[ \Gamma (\alpha _{2})\beta _{1}^{\alpha _{1}}\left/ \Gamma
(\alpha _{1})\beta _{2}^{\alpha _{2}}\right. \right] $%
\end{tabular}
\\[0.45cm] 
Binomial & $\log \left[ p_{1}^{\mu _{1}}(1-p_{2})^{\mu _{1}-n}\left/
p_{2}^{\mu _{1}}(1-p_{1})^{\mu _{1}-n}\right. \right] $ \\[0.15cm] 
Poisson & $\lambda _{1}-\lambda _{2}+\lambda _{2}\log \frac{\lambda _{2}}{%
\lambda _{1}}.$%
\end{tabular}%
\end{equation*}

\bigskip

\bibliographystyle{plain}
\bibliography{main}

\end{document}